\documentclass{amsart}
\usepackage[latin1]{inputenc}
\usepackage{hyperref}
\usepackage{texdraw}
\usepackage{amssymb}
\usepackage{color}
\usepackage{graphicx}

\makeatletter

\newfont{\gothic}{eufm10 scaled 1100}

\theoremstyle{plain}
\newtheorem{thm}{Theorem}[section]
\numberwithin{figure}{section} %% Comment out for sequentially-numbered
\theoremstyle{plain}
\newtheorem{cor}[thm]{Corollary} %%Delete [thm] to re-start numbering
\theoremstyle{plain}
\newtheorem{conj}[thm]{Conjecture} %%Delete [thm] to re-start numbering
\theoremstyle{plain}
\theoremstyle{plain}
\newtheorem{lem}[thm]{Lemma} %%Delete [thm] to re-start numbering
\theoremstyle{plain}
\newtheorem{prop}[thm]{Proposition} %%Delete [thm] to re-start numbering
\theoremstyle{plain}
\newtheorem{Def}[thm]{Definition} %%Delete [thm] to re-start numbering
\theoremstyle{remark}
\newtheorem{rem}[thm]{Remark}
\theoremstyle{remark}

\usepackage[arrow,matrix,curve,ps]{xy}
\CompileMatrices

\makeatother

\begin{document}

\title{Topological equivalence of holomorphic foliation germs of rank $1$ with isolated singularity in the Poincar\'e domain}

\date{\today}

\author{Thomas Eckl and Michael Loenne}



%\thanks{}

\address{Thomas Eckl, Department of Mathematical Sciences, The University of Liverpool, Mathematical
               Sciences Building, Liverpool, L69 7ZL, England, U.K.}

\email{thomas.eckl@liv.ac.uk}

\address{Michael Loenne, Mathematisches Institut, Universit{\"a}t Bayreuth, Universit{\"a}tsstra{\ss}e 30, 95447 Bayreuth,    
               Germany}

\email{Michael.Loenne@uni-bayreuth.de}

%\urladdr{http://pcwww.liv.ac.uk/~eckl/}

\maketitle

\begin{abstract}
We show that the topological equivalence class of holomorphic foliation germs with an isolated singularity of Poincar\'e type is determined by the topological equivalence class of the real intersection foliation of the (suitably normalized) foliation germ with a sphere centered in the singularity. We use this Reconstruction Theorem to completely classify topological equivalence classes of plane holomorphic foliation germs of Poincar\'e type and discuss a conjecture on the classification in dimension $\geq 3$.
\end{abstract}

%\tableofcontents

\pagestyle{myheadings}
\markboth{THOMAS ECKL AND MICHAEL LOENNE}{Topological equivalence of holomorphic foliation germs of rank $1$}

\setcounter{section}{-1}

\section{Introduction}

\noindent As for isolated singularities of analytic set germs (see \cite{BK86} in the case of plane curve germs) a standard technique to study the topology of holomorphic foliation germs with isolated singularity looks at the intersection of their integral manifolds with spheres centered in the origin (see~\cite{LS11} for a more general Morse-theoretic approach). The technique was particularly successful when analyzing holomorphic foliation germs represented by vector fields, that is, foliation germs with $1$-dimensional leaves: Guckenheimer~\cite{Guc72} and Camacho, Kuiper and Palis~\cite{CKP78} (who use polycylinders instead of spheres) classify foliation germs represented by generic linearizable vector fields, whereas Camacho and Sad~\cite{CS82b} treat resonant cases of plane foliation germs represented by holomorphic vector fields of Siegel type.

\noindent Note that Ito \cite{Ito94} and Ito and Scardua \cite{IS05} investigate a kind of reverse situation: They show that a holomorphic vector field everywhere transversal to a sphere has exactly one singularity in the ball bounded by the sphere, but they do not further investigate the intersection foliation and its relation to the original holomorphic vector field.

\noindent In this paper we first prove a reconstruction theorem for holomorphic foliation germs represented by a vector field of Poincar\'e type, that is, the linear part of the vector field has eigenvalues whose convex hull in $\mathbb{C}$ does not contain $0 \in \mathbb{C}$: The topological equivalence class of such a holomorphic foliation germ is uniquely determined by the real-analytic foliation obtained on a sphere around the singularity when intersecting it with all the leaves of a holomorphically equivalent, normalized foliation germ. One direction of this topological equivalence was already observed by Arnol$'$d \cite{Arn69}; he showed that the intersection foliation is topologically equivalent to the cone foliation over the intersection foliation. For more details on the other direction see Thm.~\ref{topeq-intersec-fol-thm} and the preceding discussion in sections 1 and 2.

\noindent A similar reconstruction theorem for holomorphic foliation germs represented by vector fields of Siegel type (that is, not of Poincar\'e type and the linear part has only non-zero eigenvalues) seems possible. The main obstacles to prove such a theorem are missing normal forms and the fact that leaves of such foliation germs may not intersect spheres around the singularity transversally, but tangentially. However, in sufficiently normal situations the intersection of leaves and sphere still combine to a real-analytic foliation on the sphere, and the tangential locus is the polar variety of Lim\'on and Seade~\cite{LS11} with useful properties.

\noindent In sections 3 to 6 we use the Reconstruction Theorem~\ref{topeq-intersec-fol-thm} to completely classify topological equivalence classes of plane holomorphic foliation germs represented by vector fields of Poincar\'e type. Some of the cases in this classification are well-known, for example by Guckenheimer's Stability Theorem \cite{Guc72} on foliation germs given by vector fields whose linear part has $\mathbb{R}$-linearly independent eigenvalues. Nevertheless we describe the topology of the associated intersection foliations in these cases in full detail because this description is missing in the literature and may be useful for classification in higher dimensions.  In Section~\ref{higher-dim-sec} we speculate how to extend the $2$-dimensional picture and Guckenheimer's Stability to higher-dimensional foliation germs represented by vector fields of Poincar\'e type.

\noindent Recently, Mar\'{\i}n and Mattei presented a classification of topological equivalence classes of plane holomorphic foliation germs satisfying weak genericity assumptions \cite{MaMa12}, by exhibiting an invariant based on the reduction of plane holomorphic foliation singularities and the holonomy around irreducible exceptional components of the reduction. However this classification does not cover the resonant cases discussed in Section~\ref{resonant-sec} because these are not of generic general type, in the terminology of \cite{MaMa12} (see Rem.~\ref{not-GT-rem}). In any case, Mar\'{\i}n and Mattei give no explicit lists of topological equivalent foliation germs. 

\noindent Similarly, Ilyashenko and Yakovenko do not use their calculation of holonomy along the unique closed leaf of resonant foliation germs \cite[27.C]{IY08} to classify these germs.

\section{Preliminaries on holomorphic foliation germs of rank 1} \label{prelim-sec}

\noindent In this section we collect well-known definitions and and results on holomorphic foliations with as much details as necessary to fix notations and certain choices simplifying the later arguments. In essence, all of the following can be found in the monograph \cite{IY08}.

\begin{Def}
A germ of a holomorphic foliation of rank $1$ in $\mathbb{C}^n$ with an isolated singularity in $0 \in \mathbb{C}^n$ is an equivalence class of pairs $[U, \theta]$ where $U \subset \mathbb{C}^n$ is an open neighborhood of $0$ with holomorphic coordinates $z_1, \ldots, z_n$ and
\[ \theta = f_1 \frac{\partial}{\partial z_1} + \ldots + f_n \frac{\partial}{\partial z_n} \]
is a holomorphic vector field such that $f_1, \ldots, f_n \in \mathcal{O}(U)$ vanish simultaneously only in $0$.

\noindent Two such pairs $[U, \theta]$ and $[U^\prime, \theta^\prime]$ are equivalent if there exists an open neighborhood $V \subset U \cap U^\prime$ of $0 \in \mathbb{C}^n$ and a function $h \in \mathcal{O}^\ast(V)$ such that
\[ h \cdot \theta_{|V} = \theta^\prime_{|V}. \]
\end{Def}

\noindent We denote such holomorphic foliation germs by $\mathcal{F}$.

\noindent This definition is equivalent to other definitions of holomorphic foliation germs, see the discussion in \cite[I.2]{IY08}. In particular, if $[U, \theta]$ represents a holomorphic foliation germ $\mathcal{F}$ of rank 1 in $\mathbb{C}^n$ with an isolated singularity in $0 \in \mathbb{C}^n$ then for all $p \in U -\{0\}$ there exists an open neighborhood $V \subset U$ of $p$ such that the foliation restricted to $V$ is the \textit{standard holomorphic foliation} in suitable holomorphic coordinates $w_1, \ldots, w_n$ on $V$ centred in $p$, that is, 
\[ \theta(w_1) = \ldots = \theta(w_{n-1}) = 0. \]
Furthermore, the integral curves of $\theta$ in $U$ are covered by the local leaves or \textit{plaques} of these local standard holomorphic foliations given by the intersection of the level sets of $w_1, \ldots, w_{n-1}$. These integral curves are called the \textit{leaves} of $\mathcal{F}$.

\noindent We will consider the following topological equivalence relation on holomorphic foliation germs:

\begin{Def} \label{TopEq-def}
Two  holomorphic foliation germs $\mathcal{F}$, $\mathcal{F}^\prime$ of rank $1$ with an isolated singularity in  $0 \in \mathbb{C}^n$ and represented by $[U, \theta]$, $[U^\prime, \theta^\prime]$ are called topologically equivalent if there exists a homeomorphism $\phi: V \rightarrow V^\prime$ of open neighborhoods $V \subset U$, $V^\prime \subset U^\prime$ of $0$ such that $\phi(0) = 0$ and the leaves of $\mathcal{F}$ in $V$ are mapped onto the leaves of  $\mathcal{F}^\prime$ in $V^\prime$ by $\phi$.
\end{Def}

\noindent If $\phi$ is biholomorphic we say that $\mathcal{F}$ and $\mathcal{F}^\prime$ are \textit{holomorphically equivalent}.

\noindent We will focus on a special type of holomorphic foliation germs:

\begin{Def}
A holomorphic foliation germ of rank $1$ with an isolated singularity in  $0 \in \mathbb{C}^n$ represented by $[U, \theta]$  is said to be of Poincar\'e type if the eigenvalues $\lambda_1, \ldots, \lambda_n \in \mathbb{C}$ of the linear part
\[ A = \left( \frac{\partial f_j}{\partial z_i}(0) \right)_{i,j=1, \ldots, n} \]
of $\theta = \sum_{j=1}^n f_j \frac{\partial}{\partial z_j}$ generate a convex hull not containing $0 \in \mathbb{C}$. Then the tuple of eigenvalues $(\lambda_1, \ldots, \lambda_n) \in \mathbb{C}^n$ is said to be in the Poincar\'e domain.
\end{Def}

\noindent The classical theorems of Poincar\'e and Poincar\'e-Dulac (see~\cite[\S 24.D and E]{Arn83}) state that all  holomorphic foliation germs of rank $1$ with an isolated singularity in  $0 \in \mathbb{C}^n$ of Poincar\'e type are even holomorphically equivalent to such germs $\mathcal{F}$ of Poincar\'e type represented by an open subset $U \subset \mathbb{C}^n$ containing $0$ and a holomorphic vector field $\theta = \sum_{i=1}^n f_i(z) \frac{\partial}{\partial z_i}$ for which the following hold:
\begin{enumerate}
\item[(i)] In $U$, the $f_i(z)$ can be developed into powers series in the variables $z_1, \ldots, z_n$.
\item[(ii)] The linear part $ A = \left( \frac{\partial f_j}{\partial z_i}(0) \right)_{i,j=1, \ldots, n}$ of $\theta$ is in Jordan normal form.
\item[(iii)] If $\lambda_1, \ldots, \lambda_n$ are the eigenvalues of $A$ appearing with their algebraic multiplicity, then the non-vanishing monomials $z_1^{m_1} \cdots z_n^{m_n}$ in $f_i(z)$, with $m_i \in \mathbb{N}_0$, satisfy 
\[ \lambda_i = \sum_{j=1}^n m_j \lambda_j. \]
\end{enumerate}

\noindent Note that condition (ii) implies that the linear term in $f_i(z)$ has the form $\lambda_iz_i$ or $\lambda_iz_i +  z_{i+1}$. In the latter case $\lambda_i = \lambda_{i+1}$ by the properties of the Jordan normal form, hence all the non-vanishing monomials in  these linear terms satisfy condition (iii). For $m = (m_1, \ldots, m_n) \in \mathbb{N}^n_0$ a relation 
$\lambda_i = \sum_{j=1}^n m_j \lambda_j$ is called a \textit{resonance} of the eigenvalues $\lambda_1, \ldots, \lambda_n$, and the monomial $z^m := z_1^{m_1} \cdots z_n^{m_n}$ is called a \textit{resonant monomial} if it appears in $f_i(z)$.

\begin{rem} \label{PD-res-rem}
If  $(\lambda_1, \ldots, \lambda_n) \in \mathbb{C}^n$ is in the Poincar\'e domain then there are only finitely many resonances $\lambda_i = \sum_{j=1}^n m_j \lambda_j$. Furthermore, a resonance relation with $\lambda_i$ on the left hand side is either the \textit{trivial resonance relation} $\lambda_i = \lambda_i$, or $\lambda_i$ does not appear on the right hand side at all, that is $m_i=0$. For proofs, see~\cite[\S 24.B]{Arn83}. Note finally that we do not require $\sum_i m_i \geq 2$ as in \cite{Arn83} but only distinguish between the trivial resonant monomial $z_i$ in $f_i(z)$ and non-trivial resonant monomials.
\end{rem}

\begin{rem} \label{PD-est-rem}
Let $\mathcal{F}$ be a holomorphic foliation germ satisfying (i), (ii) and (iii). For the tuple $(\lambda_1, \ldots, \lambda_n) \in \mathbb{C}^n$ in the Poincar\'e domain there exists a maximal real constant $c > 0$ such that 
\[ |\sum_{i=1}^n \lambda_i t_i| \geq c \cdot \sum_{i=1}^n t_i, \]
for all real numbers $t_1, \ldots, t_n \geq 0$. The number $c$ can be interpreted as the distance of the convex hull of $\lambda_1, \ldots, \lambda_n$ in $\mathbb{C}$ from $0$. By separately rescaling the coordinates we can achieve that the entries of the matrix A on the superdiagonal are arbitrarily small. If the entries are $\frac{c}{2n}$ we will call $\mathcal{F}$ \textit{normalized} (see the next section).
\end{rem}

\begin{rem}  \label{dim2-fol-rem}
If $n = 2$ then every normalized holomorphic foliation germ of rank $1$ with an isolated singularity in  $0 \in \mathbb{C}^n$ of Poincar\'e type is represented by a vector field of one of the following types:
\begin{itemize}
\item[(1)] $\theta = \lambda z_1 \frac{\partial}{\partial z_1} + z_2 \frac{\partial}{\partial z_2}$, where $\lambda \in \mathbb{C} - \mathbb{R}$,
\item[(2)] $\theta = \lambda z_1 \frac{\partial}{\partial z_1} + z_2 \frac{\partial}{\partial z_2}$, where $\lambda \in \mathbb{R}_{>0}$,
\item[(3)] $\theta = (m z_1 + z_2^m) \frac{\partial}{\partial z_1} + z_2 \frac{\partial}{\partial z_2}$, where $m \geq 2$, or
\item[(4)] $\theta = (z_1 + \frac{1}{4}z_2) \frac{\partial}{\partial z_1} + z_2 \frac{\partial}{\partial z_2}$ because the constant $c$ of Rem.~\ref{PD-est-rem} is $1$ in this case.
\end{itemize}
\end{rem}

\section{The intersection foliation} \label{intsectfol-sec}

\noindent In this section  $S^{2n-1}_{\epsilon}$ denotes the (real) $2n-1$-dimensional sphere in $\mathbb{C}^n$ centered in $0 \in \mathbb{C}^n$ with radius $\epsilon$, and $B^{2n}_{\epsilon}$ denotes the (real) $2n$-dimensional ball in $\mathbb{C}^n$ centered in $0 \in \mathbb{C}^n$ with radius $\epsilon$.

\noindent $\mathcal{F}$ is always a normalized holomorphic foliation germ of rank $1$ with an isolated singularity in  $0 \in \mathbb{C}^n$ of Poincar\'e type, and $[U, \theta]$ represents $\mathcal{F}$, with $\theta = \sum_{i=1}^n f_i(z)\frac{\partial}{\partial z_i}$. Furthermore, $\lambda_1, \ldots, \lambda_n$ are the eigenvalues of the linear part of $\theta$ appearing with their algebraic multiplicity, and $c>0$ is the real constant for the tuple $(\lambda_1, \ldots, \lambda_n) \in \mathbb{C}^n$ in the Poincar\'e domain introduced in Rem.~\ref{PD-est-rem}. 

\noindent Arnol$'$d \cite[Thm.~5]{Arn69} observed that the leaves of holomorphic foliation germs $\mathcal{F}$ as above intersect spheres $S^{2n-1}_\epsilon$ with small enough radius $\epsilon$ everywhere transversally. In particular, in each point $p \in S^{2n-1}_\epsilon$ the (real) tangent spaces of the leaf of $\mathcal{F}$ through $p$ and of $ S^{2n-1}_{\epsilon}$ intersect in a (real) $1$-dimensional subspace. This yields a $1$-dimensional distribution on the real $\mathcal{C}^\infty$-manifold $ S^{2n-1}_{\epsilon}$ denoted by $\mathcal{F} \cap  S^{2n-1}_{\epsilon}$. This distribution is integrable because it is $1$-dimensional, see~\cite{War83}. Therefore we obtain a real foliation on $ S^{2n-1}_{\epsilon}$ with $1$-dimensional leaves , also denoted by $\mathcal{F} \cap  S^{2n-1}_{\epsilon}$ and called the \textit{real intersection foliation} or \textit{trace foliation} of $\mathcal{F}$ with $S^{2n-1}_{\epsilon}$. 

\noindent Its leaves can be canonically oriented: In each point $p \in S^{2n-1}_{\epsilon}$ choose those vectors in the tangent subspace given by the distribution $\mathcal{F} \cap  S^{2n-1}_{\epsilon}$ in $p$ which together with the tangent vectors of the leaf of $\mathcal{F}$ through $p$ pointing away from $0 \in \mathbb{C}^n$ represent the positive orientation of the complex structure on the leaf. Taking in each point $p \in S^{2n-1}_{\epsilon}$ a unit vector oriented in that way yields a nowhere vanishing vector field on $S^{2n-1}_{\epsilon}$ whose flow, denoted by $\Phi_{\mathcal{F}}$, has integral curves coinciding with the leaves of $\mathcal{F} \cap  S^{2n-1}_{\epsilon}$.

\begin{Def}
Two real $1$-dimensional foliations $\mathcal{F}, \mathcal{G}$ on the sphere $S^{2n-1}$ are called topologically equivalent if there exists a homeomorphism $\phi: S^{2n-1} \rightarrow S^{2n-1}$ mapping the leaves of $\mathcal{F}$ onto the leaves of $\mathcal{G}$.  
\end{Def}  

\noindent In \cite{Arn69}, Arnol$'$d continued to show that for $\epsilon$ small enough the foliation $\mathcal{F} \cap B^{2n-1}_\epsilon$ is topologically equivalent to the foliation of $B^{2n-1}_\epsilon$ by the cones over the leaves of $\mathcal{F} \cap S^{2n-1}_\epsilon$ with vertex $0 \in \mathbb{C}^n$, and this equivalence is realised by a homeomorphism of $B^{2n-1}_\epsilon$ on itself. This obviously implies that topological equivalence of $\mathcal{F}$ and $\mathcal{G}$ follows from topological equivalence of $\mathcal{F} \cap S^{2n-1}_\epsilon$ and $\mathcal{G} \cap S^{2n-1}_\epsilon$.

\noindent It seems to be a well-accepted fact that the converse is also true, but because of lack of reference (and in particular of a recipe to construct the topological equivalence of the intersection foliations) we decided to present this proof and construction in all details.

\begin{thm}[Reconstruction Theorem] \label{topeq-intersec-fol-thm}
Two normalized  holomorphic foliation germs $\mathcal{F}, \mathcal{G}$ with an isolated singularity in  $0 \in \mathbb{C}^n$ of Poincar\'e type are topologically equivalent if, and only if the real intersection foliations $\mathcal{F} \cap  S^{2n-1}_{\epsilon}$ and  $\mathcal{G} \cap  S^{2n-1}_{\epsilon}$, $0 < \epsilon \ll 1$, are topologically equivalent. 
\end{thm}
\begin{proof}
As discussed above Arnol$'$d showed that $\mathcal{F} \cap B^{2n-1}_\epsilon$ resp. $\mathcal{G} \cap B^{2n-1}_\epsilon$ is topologically equivalent to the foliation $\widehat{\mathcal{F}}$ resp. $\widehat{\mathcal{G}}$ of $B^{2n-1}_\epsilon$ by the cones over the leaves of $\mathcal{F} \cap  S^{2n-1}_{\epsilon}$ resp. $\mathcal{G} \cap  S^{2n-1}_{\epsilon}$ with vertex $0 \in \mathbb{C}^n$. As already noticed that shows one direction of the Theorem.

\noindent To prove that a topological equivalence of $\mathcal{F}$ and $\mathcal{G}$ induces a topological equivalence of the intersection foliations $\mathcal{F} \cap  S^{2n-1}_{\epsilon}$ and $\mathcal{G} \cap  S^{2n-1}_{\epsilon}$, it is enough to construct the latter one from a topological equivalence of the cone foliations $\widehat{\mathcal{F}}$ and $\widehat{\mathcal{G}}$, in the sense of Def.~\ref{TopEq-def}.  

\noindent By possibly decreasing $\epsilon$ to $\epsilon^\prime$  this topological equivalence can be realised by an embedding $H_C: B^{2n}_{\epsilon^\prime} \hookrightarrow B^{2n}_{\epsilon}$ such that $H_C(0) = 0$ and leaves of the foliation cone over $\widehat{\mathcal{F}} \cap S^{2n-1}_{\epsilon^\prime}$ are mapped into leaves  of the foliation cone over $\widehat{\mathcal{G}} \cap S^{2n-1}_{\epsilon}$. Composing $H_C$ with the radial projection $r: B^{2n}_{\epsilon} - \{0\} \rightarrow S^{2n-1}_{\epsilon}$ produces a continuous map
\[ h: S^{2n-1}_{\epsilon^\prime} \stackrel{H_C}{\hookrightarrow}  B^{2n}_{\epsilon} - \{0\} \stackrel{r}{\rightarrow} 
        S^{2n-1}_{\epsilon}  \]
Since $H_C$ may map $S^{2n-1}_{\epsilon^\prime}$ to a topological manifold in $ B^{2n}_{\epsilon}$ intersecting the same radial line more than once, $h$ may not be injective, and hence not the wanted homeomorphism. But
from $h$ we will be able to construct a homeomorphism $g:  S^{2n-1}_{\epsilon^\prime} \rightarrow  S^{2n-1}_{\epsilon}$ defining a topological equivalence of $\widehat{\mathcal{F}} \cap S^{2n-1}_{\epsilon^\prime}$ with $\widehat{\mathcal{G}} \cap S^{2n-1}_{\epsilon}$. This shows the theorem since the cone structure of $\widehat{\mathcal{F}}$ implies that $\widehat{\mathcal{F}} \cap S^{2n-1}_{\epsilon^\prime}$ is topologically equivalent to $\widehat{\mathcal{F}} \cap S^{2n-1}_{\epsilon}$, and by construction, $\widehat{\mathcal{F}} \cap S^{2n-1}_{\epsilon} = \mathcal{F} \cap S^{2n-1}_{\epsilon}$ and $\widehat{\mathcal{G}} \cap S^{2n-1}_{\epsilon} = \mathcal{G} \cap S^{2n-1}_{\epsilon}$.   

\noindent Let $L_{\widehat{\mathcal{F}},x}$ denote the leaf of $\widehat{\mathcal{F}} \cap S^{2n-1}_{\epsilon^\prime}$ through $x \in S^{2n-1}_{\epsilon^\prime}$, and $L_{\widehat{\mathcal{G}},y}$ the leaf of $\widehat{\mathcal{G}} \cap S^{2n-1}_{\epsilon}$ through $y \in S^{2n-1}_{\epsilon}$. Let $C_{\widehat{\mathcal{F}},x}$ denote the radial cone in $B^{2n}_{\epsilon^\prime}$ with vertex in $0$ over the leaf $L_{\widehat{\mathcal{F}},x} \subset S^{2n-1}_{\epsilon^\prime}$, and similarly $C_{\widehat{\mathcal{G}},y}$ the radial cone in $B^{2n}_{\epsilon}$ with vertex in $0$ over the leaf $L_{\widehat{\mathcal{G}},y} \subset S^{2n-1}_{\epsilon}$. 

\vspace{0.2cm}

\noindent \textit{Claim 1.} $h$ is a surjective map, and $h(L_{\widehat{\mathcal{F}},x}) = L_{\widehat{\mathcal{G}},h(x)}$ for each $x \in S^{2n-1}_{\epsilon^\prime}$.
\begin{proof}
There exists $\epsilon^{\prime\prime} \ll \epsilon$ such that $B^{2n}_{\epsilon^{\prime\prime}} \subset H_C(B^{2n}_{\epsilon^\prime})$, as $H_C$ is an embedding fixing $0$. Consequently, for every $y \in S^{2n-1}_{\epsilon}$, the segment $[y,0] \subset B^{2n}_{\epsilon}$ intersects $H_C(B^{2n}_{\epsilon^\prime})$ in a point $H_C(x)$, with $x \in S^{2n-1}_{\epsilon^\prime}$. Hence $h(x) = y$, and the surjectivity of $h$ is shown. 

\noindent The equality $L_{\widehat{\mathcal{G}},h(x)} = h(L_{\widehat{\mathcal{F}},x})$ follows from the fact that by definition, the topological equivalence $H_C$ maps $C_{\widehat{\mathcal{F}},x}$ bijectively onto $C_{\widehat{\mathcal{G}},h(x)} \cap H_C(B^{2n}_{\epsilon^\prime})$. 
\end{proof} 

\noindent We want to relate $h$ to the flows $\Phi_{\widehat{\mathcal{F}}}: S^{2n-1}_{\epsilon^\prime} \times \mathbb{R} \rightarrow S^{2n-1}_{\epsilon^\prime}$ and $\Phi_{\widehat{\mathcal{G}}}: S^{2n-1}_{\epsilon} \times \mathbb{R} \rightarrow S^{2n-1}_{\epsilon}$ whose integral curves are the leaves of the real intersection foliations $\widehat{\mathcal{F}} \cap S^{2n-1}_{\epsilon^\prime}$ and $\widehat{\mathcal{G}} \cap S^{2n-1}_{\epsilon}$. Note that in general, $\Phi_{\widehat{\mathcal{F}}}$ and $\Phi_{\widehat{\mathcal{G}}}$ will not commute with $h$, that is
\[ (h \circ \Phi_{\widehat{\mathcal{F}}})(x,t) \neq \Phi_{\widehat{\mathcal{G}}}(h(x),t). \]
To obtain the correct relation, we lift $\Phi_{\widehat{\mathcal{F}}}$ to a flow $\Phi_{\widetilde{\mathcal{F}}}$ on $S^{2n-1}_{\epsilon^\prime} \times \mathbb{R}$, by setting
\[ \Phi_{\widetilde{\mathcal{F}}}((x,t^\prime), t) := (\Phi_{\widehat{\mathcal{F}}}(x, t), t+t^\prime), x \in S^{2n-1}_{\epsilon^\prime},    
    t, t^\prime \in \mathbb{R}.  \]
The integral curves of $\Phi_{\widetilde{\mathcal{F}}}$ define a foliation $\widetilde{\mathcal{F}}$ on $S^{2n-1}_{\epsilon^\prime} \times \mathbb{R}$ whose leaves project onto the leaves of $\widehat{\mathcal{F}}$ on $S^{2n-1}_{\epsilon^\prime}$. Similarly, 
\[ \Phi_{\widetilde{\mathcal{G}}}((y,s^\prime), s) := (\Phi_{\widehat{\mathcal{G}}}(y, s), s+s^\prime), y \in S^{2n-1}_{\epsilon},    
    s, s^\prime \in \mathbb{R}  \]
defines a flow $\Phi_{\widetilde{\mathcal{G}}}$ and a foliation $\widetilde{\mathcal{G}}$ on $S^{2n-1}_{\epsilon} \times \mathbb{R}$ whose leaves project onto the leaves of $\widehat{\mathcal{G}}$ on $S^{2n-1}_{\epsilon}$.

\noindent Let $p_1, p_2$ resp.\ $q_1, q_2$ denote the projections from $S^{2n-1}_{\epsilon^\prime} \times \mathbb{R}$ resp.\  $S^{2n-1}_{\epsilon} \times \mathbb{R}$ to the first and second component. If $U \subset S^{2n-1}_{\epsilon^\prime}$ resp. $V \subset S^{2n-1}_{\epsilon}$ are foliation charts of $\widehat{\mathcal{F}}$ resp.\ $\widehat{\mathcal{G}}$, then $p_1^{-1}(U)$ resp.\ $q_1^{-1}(V)$ are foliation charts of $\widetilde{\mathcal{F}}$ resp. $\widetilde{\mathcal{G}}$. Consequently, $h: S^{2n-1}_{\epsilon^\prime} \rightarrow S^{2n-1}_{\epsilon}$ can be lifted exactly in one way to a continuous map
\[ \widetilde{H}: S^{2n-1}_{\epsilon^\prime} \times \mathbb{R}\rightarrow S^{2n-1}_{\epsilon} \times \mathbb{R} \]
such that $\widetilde{H}(x,0) = (h(x),0)$ for all $x \in S^{2n-1}_{\epsilon^\prime}$ and the leaves of $\widetilde{\mathcal{F}}$ are mapped into the leaves of $\widetilde{\mathcal{G}}$. In particular, $q_1 \circ \widetilde{H} = h \circ p_1$, and since $(\Phi_{\widehat{\mathcal{F}}}(x,t),t) = \Phi_{\widetilde{\mathcal{F}}}((x,0), t)$ the point $\widetilde{H}(\Phi_{\widehat{\mathcal{F}}}(x,t),t) \in S^{2n-1}_{\epsilon} \times \mathbb{R}$ must be in the same $\widetilde{G}$-leaf as $\widetilde{H}(x,0) = (h(x),0)$. Hence there is an $s \in \mathbb{R}$ such that
\[  \Phi_{\widetilde{\mathcal{G}}}((h(x),0), s) = \widetilde{H}(\Phi_{\widehat{\mathcal{F}}}(x,t),t), \]
and the defining equations of $\Phi_{\widetilde{\mathcal{G}}}$ and $\widetilde{H}$ imply that
\[ \Phi_{\widehat{\mathcal{G}}}((h(x),s), s) = (h(\Phi_{\widehat{\mathcal{F}}}(x,t)), (q_2 \circ \widetilde{H})(\Phi_{\widehat{\mathcal{F}}}(x,t),t)). \]
Setting $\tau := q_2 \circ \widetilde{H} \circ (\Phi_{\widehat{\mathcal{F}}} \times p_2): S^{2n-1}_{\epsilon^\prime} \times \mathbb{R}\rightarrow \mathbb{R}$ and comparing the second and the first components yield $s = \tau(x,t)$ and
\begin{equation} \label{def-tau-eq}
h(\Phi_{\widehat{\mathcal{F}}}(x,t)) = \Phi_{\widehat{\mathcal{G}}}((h(x), \tau(x,t)).
\end{equation}
This is the requested relation between $h$, $\Phi_{\widehat{\mathcal{F}}}$ and $\Phi_{\widehat{\mathcal{G}}}$.

\noindent By construction we have
\begin{equation}
\tau(x,0) = 0.
\end{equation}
To obtain further properties of $\tau$ we need to investigate the leaves $L_{\widehat{\mathcal{F}},x}$ and $L_{\widehat{\mathcal{G}},y}
$ of the real intersection foliations $\widehat{\mathcal{F}} \cap S^{2n-1}_{\epsilon^\prime}$ and $\widehat{\mathcal{G}} \cap S^{2n-1}_{\epsilon}$ in more details. First of all, we must carefully distinguish between the \textit{leaf topology} on $L_{\widehat{\mathcal{F}},x} = \Phi_{\widehat{\mathcal{F}}}(\{x\} \times \mathbb{R})$ and $L_{\widehat{\mathcal{G}},y} = \Phi_{\widehat{\mathcal{G}}}(\{y\} \times \mathbb{R})$ defined as the finest topology such that $\Phi_{\widehat{\mathcal{F}}|\{x\} \times \mathbb{R}}$ resp. $\Phi_{\widehat{\mathcal{G}}|\{y\} \times \mathbb{R}}$ are continuous, and the \textit{inclusion topology} induced by the inclusion in $S^{2n-1}_{\epsilon^\prime}$ resp. $S^{2n-1}_{\epsilon}$. The leaf topology is always finer than the inclusion topology, and the two topologies only coincide if the leaf is locally closed in $S^{2n-1}_{\epsilon^\prime}$ resp. $ S^{2n-1}_{\epsilon}$. If $L_{\widehat{\mathcal{F}},x}$ resp. $L_{\widehat{\mathcal{G}},y}$ are not bijective images of $\{x\} \times \mathbb{R}$ resp. $\{y\} \times \mathbb{R}$ under $\Phi_{\widehat{\mathcal{F}}}$ resp. $\Phi_{\widehat{\mathcal{G}}}$ then $\Phi_{\widehat{\mathcal{F}}|\{x\} \times \mathbb{R}}$ resp. $\Phi_{\widehat{\mathcal{G}}|\{y\} \times \mathbb{R}}$ are periodic maps, and the images $L_{\widehat{\mathcal{F}},x}$ resp. $L_{\widehat{\mathcal{G}},y}$ are compact both in leaf topology and inclusion topology. In particular, in that case $L_{\widehat{\mathcal{F}},x}$ resp. $L_{\widehat{\mathcal{G}},y}$ are embedded circles in $S^{2n-1}_{\epsilon^\prime}$ resp. $ S^{2n-1}_{\epsilon}$. Note also that $\Phi_{\widehat{\mathcal{F}}|\{x\} \times \mathbb{R}}$ and $\Phi_{\widehat{\mathcal{G}}|\{y\} \times \mathbb{R}}$ are universal coverings of $L_{\widehat{\mathcal{F}},x}$ and $L_{\widehat{\mathcal{G}},y}$ endowed with the leaf topology.

\vspace{0.2cm}

\noindent \textit{Claim 2.} The leaf $L_{\widehat{\mathcal{F}},x} \subset S^{2n-1}_{\epsilon^\prime}$ is an embedded circle if, and only if the leaf $L_{\widehat{\mathcal{G}},h(x)} \subset  S^{2n-1}_{\epsilon}$ is an embedded circle. 
\begin{proof}
Assume that $L_{\widehat{\mathcal{F}},x}$ is an embedded circle, hence compact. Since $h(L_{\widehat{\mathcal{F}},x}) = L_{\widehat{\mathcal{G}},h(x)}$ by Claim 1 and $h$ is continuous in the inclusion topology, $L_{\widehat{\mathcal{G}},h(x)}$ must be compact, hence closed. Then leaf and inclusion topology on $L_{\widehat{\mathcal{G}},h(x)}$ coincide, so 
$L_{\widehat{\mathcal{G}},h(x)}$ cannot be homeomorphic to $\mathbb{R}$ in leaf topology. Consequently, $L_{\widehat{\mathcal{G}},h(x)}$ is an embedded circle. 

\noindent On the other hand, if $L_{\widehat{\mathcal{G}},h(x)}$ is an embedded circle then the cone leaf $C_{\widehat{\mathcal{G}},h(x)}$ and hence the intersection $C_{\widehat{\mathcal{G}},h(x)} \cap H_C(S^{2n-1}_{\epsilon^\prime})$ is compact. But the topological equivalence $H_C^{-1}$ maps $C_{\widehat{\mathcal{G}},h(x)} \cap (H_C(S^{2n-1}_{\epsilon^\prime})$ onto $L_{\widehat{\mathcal{F}},x}$. So $L_{\widehat{\mathcal{F}},x}$ is compact in the inclusion topology, hence compact in the coinciding leaf topology, and hence an embedded circle, not a line. 
\end{proof}

\noindent Using (\ref{def-tau-eq}) and the functorial property of the flows $\Phi_{\widehat{\mathcal{F}}}$ and $\Phi_{\widehat{\mathcal{G}}}$ we calculate
\begin{eqnarray*}
\Phi_{\widehat{\mathcal{G}}}(h(x), \tau(x,t) + \tau(\Phi_{\widehat{\mathcal{F}}}(x,t),t^\prime)) & = & \Phi_{\widehat{\mathcal{G}}}(\Phi_{\widehat{\mathcal{G}}}(h(x), \tau(x,t)),  \tau(\Phi_{\widehat{\mathcal{F}}}(x,t),t^\prime)) = \\
 & = & \Phi_{\widehat{\mathcal{G}}}(h(\Phi_{\widehat{\mathcal{F}}}(x,t)), \tau(\Phi_{\widehat{\mathcal{F}}}(x,t),t^\prime)) = \\
 & = & h(\Phi_{\widehat{\mathcal{F}}}(\Phi_{\widehat{\mathcal{F}}}(x,t), t^\prime)) = h(\Phi_{\widehat{\mathcal{F}}}(x, t+t^\prime)) = \\
 & = & \Phi_{\widehat{\mathcal{G}}}(h(x), \tau(x, t + t^\prime)).  
\end{eqnarray*}
If $\Phi_{\widehat{\mathcal{G}}}(h(x), \cdot)$ is injective this implies
\begin{equation} \label{funct-tau-eq}
\tau(x, t+t^\prime) = \tau(x,t) + \tau(\Phi_{\widehat{\mathcal{F}}}(x,t),t^\prime).
\end{equation}

\noindent If $L_{\widehat{\mathcal{F}},x}$ and  $L_{\widehat{\mathcal{G}},h(x)}$ are embedded circles then $\Phi_{\widehat{\mathcal{F}}|\{x\} \times \mathbb{R}}$ and $\Phi_{\mathcal{G}|\{h(x)\} \times \mathbb{R}}$ are periodic maps with periods $T_{\widehat{\mathcal{F}},x}$ and $T_{\mathcal{G},h(x)}$. Consequently,
\[ \tau(x, t+t^\prime) = \tau(x,t) + \tau(\Phi_{\widehat{\mathcal{F}}}(x,t),t^\prime) + k(t,t^\prime) \cdot T_{\mathcal{G},h(x)}, \]
where $k(t,t^\prime)$ is an integer continuously depending on $t, t^\prime$, hence a constant $k$. Setting $t = t^\prime = 0$ we obtain $k = k(0,0) = 0$ and thus (\ref{funct-tau-eq}).

\noindent In this situation, $\tau(x, \cdot): \mathbb{R} \rightarrow \mathbb{R}$ is the lifting of $h: L_{\widehat{\mathcal{F}},x} \rightarrow L_{\widehat{\mathcal{G}}, h(x)}$ to the universal coverings of the leaves along the flows $\Phi_{\widehat{\mathcal{F}}}: \{x\} \times \mathbb{R} \rightarrow L_{\widehat{\mathcal{F}},x}$ and $\Phi_{\widehat{\mathcal{G}}}: \{h(x)\} \times \mathbb{R} \rightarrow L_{\widehat{\mathcal{G}}, h(x)}$. Since liftings preserve fibers of the coverings this implies
\[ \tau(x, t + T_{\widehat{\mathcal{F}},x}) = \tau(x,t) + l \cdot T_{\widehat{\mathcal{G}}, h(x)}. \]
Since $H_C(L_{\widehat{\mathcal{F}},x})$ is an embedded circle in $C_{\widehat{\mathcal{G}}, h(x)}$ with $0$ in its interior, $h: L_{\widehat{\mathcal{F}},x} \rightarrow L_{\widehat{\mathcal{G}}, h(x)}$ is homotopic to a homeomorphism. Since furthermore $H_C$ preserves orientation, we conclude $l = 1$ and obtain:
\begin{equation} \label{per-tau-eq}
\tau(x, t + T_{\widehat{\mathcal{F}},x}) = \tau(x,t) + T_{\widehat{\mathcal{G}}, h(x)}.
\end{equation}

\noindent As a last property of $\tau$ we show:
\begin{equation} \label{lim-tau-eq}
\lim_{t \rightarrow \infty} \tau(x,t) = \infty\ \mathrm{and}\ \lim_{t \rightarrow -\infty} \tau(x,t) =-\infty:
\end{equation}
If $L_{\widehat{\mathcal{F}},x}$ and  $L_{\widehat{\mathcal{G}},h(x)}$ are embedded circles, this follows from (\ref{per-tau-eq}). Otherwise, both $\Phi_{\widehat{\mathcal{F}}|\{x\} \times \mathbb{R}}$ and $\Phi_{\widehat{\mathcal{G}}|\{h(x)\} \times \mathbb{R}}$ are bijective. In that case, for all $y \in L_{\widehat{\mathcal{G}},h(x)}$ the set of $t \in \mathbb{R}$ such that
\[ y = h(\Phi_{\widehat{\mathcal{F}}}(x,t)) = \Phi_{\widehat{\mathcal{G}}}(h(x), \tau(x,t)) \]
is bounded because the intersection of the line segment $[y,0]$ with $H_C(L_{\widehat{\mathcal{F}},x})$ equals $[y,0] \cap H_C(S^{2n-1}_{\epsilon^\prime})$, hence is compact. On the other hand, $|\tau(x,t)|$ may be arbitarily large, as $h(L_{\widehat{\mathcal{F}},x}) = L_{\widehat{\mathcal{G}},h(x)}$. Both facts together contradict $\lim_{t \rightarrow \pm \infty} |\tau(x,t)| \neq \infty$. The signs are again as claimed because $H_C$ preserves orientation.

\vspace{0.2cm}

\noindent The aim is now to modify $\tau$ to a continuous map $\sigma: S^{2n-1}_{\epsilon^\prime} \times \mathbb{R} \rightarrow \mathbb{R}$ which is strictly increasing and surjective for fixed $x \in S^{2n-1}_{\epsilon^\prime}$ but still satisfies a functorial property analogous to (\ref{funct-tau-eq}). We use $\sigma$ to modify $h$ to a topological equivalence $g$ of $\widehat{\mathcal{F}} \cap S^{2n-1}_{\epsilon^\prime}$ and $\widehat{\mathcal{G}} \cap S^{2n-1}_{\epsilon}$. 

\noindent The modification of $\tau$ to $\sigma$ and hence from $h$ to $g$ is done in two steps: First, we cut off any "moving backwards" of the image of the leaf $L_{\widehat{\mathcal{F}},x}$ on the leaf $L_{\widehat{\mathcal{G}}, h(x)}$ by keeping the map stationary whenever such a backwards move starts. Then we smoothen the stationary intervals to obtain a bijective map. For the image $H_C(L_{\widehat{\mathcal{F}},x})$ in $C_{\widehat{\mathcal{G}},h(x)}$ these steps may locally be visualized as follows:

\begin{center}
\begin{picture}(340,70)(0,0) 

\put(0,56){\line(1,0){80}}
\multiput(0,56)(0,-16){4}{\line(0,-1){8}}
\multiput(80,56)(0,-16){4}{\line(0,-1){8}}
\put(0,16){\line(6,-1){48}}
\put(48,8){\line(1,3){8}}
\put(40,16){\line(1,1){16}}
\put(40,16){\line(-1,2){8}}
\put(32,32){\line(3,2){24}}
\put(56,48){\line(1,-1){24}}
\put(24,8){\rule{0.5pt}{4pt}}
\put(20,2){\tiny{$H_C(L_{\widehat{\mathcal{F}},x})$}}
\put(72,12){\line(1,0){12}}
\put(84,10){\tiny{$C_{\widehat{\mathcal{G}},h(x)}$}}
\put(64,56){\rule{0.5pt}{4pt}}
\put(60,62){\tiny{$L_{\widehat{\mathcal{G}}, h(x)}$}}

\put(90,25){$\xrightarrow{\hspace*{0.6cm}}$}

\put(120,56){\line(1,0){80}}
\multiput(120,56)(0,-16){4}{\line(0,-1){8}}
\multiput(200,56)(0,-16){4}{\line(0,-1){8}}
\put(120,16){\line(6,-1){48}}
\put(168,8){\line(1,3){8}}
\put(176,16){\line(0,1){32}}
\put(176,48){\line(1,-1){24}}

\put(210,25){$\xrightarrow{\hspace*{0.6cm}}$}

\put(240,56){\line(1,0){80}}
\multiput(240,56)(0,-16){4}{\line(0,-1){8}}
\multiput(320,56)(0,-16){4}{\line(0,-1){8}}
\put(240,16){\line(6,-1){48}}
\put(288,8){\line(1,6){4}}
\put(292,32){\line(1,-4){4}}
\put(296,16){\line(1,4){8}}
\put(304,48){\line(2,-3){16}}

\end{picture}
\end{center}

\vspace{0.5cm}

\noindent Continuity of $\tau$ and (\ref{lim-tau-eq}) imply that $\mu(x,t) := \max_{t^\prime \leq t} \{\tau(x,t^\prime)\}$ defines a continuous function $\mu: S^{2n-1}_{\epsilon^\prime} \times \mathbb{R} \rightarrow \mathbb{R}$ which is surjective and increasing for fixed $x$. It holds that
\begin{eqnarray}
   \mu(x, t+t^\prime) & = & \max_{t^{\prime\prime} \leq t+t^\prime} \{\tau(x,t^{\prime\prime})\} = 
   \max_{t^{\prime\prime\prime} \leq t^\prime} \{\tau(x,t^{\prime\prime\prime}+t) \} =  \notag \\ \label{funct-mu-eq}
    & = & \max_{t^{\prime\prime\prime} \leq t^\prime} \{\tau(\Phi_{\widehat{\widehat{\mathcal{F}}}}(x,t),t^{\prime\prime\prime}) + \tau(x,t) \} = \\                  
    & = & \mu(\Phi_{\widehat{\mathcal{F}}}(x,t),t^\prime) + \tau(x,t). \notag 
\end{eqnarray}

\noindent $\mu(x, \cdot)$ is not necessarily strictly increasing. To modify $\mu$ to a strictly increasing function without destroying (\ref{funct-mu-eq}) we introduce the growth function
\[ \gamma_{\delta}(x,t) := \min_{t < t^\prime} \{ t^\prime: \tau(x, t^\prime) = \tau(x,t) + \delta \} - t > 0, \]
for a fixed $\delta >0$. It is continuous on $S^{2n-1}_{\epsilon^\prime} \times \mathbb{R}$, hence averaging $\mu$ by $\gamma_\delta$ leads to the continuous function
\[ \sigma(x,t) := \frac{1}{\gamma_\delta(x,t)} \int_t^{t+\gamma_\delta(x,t)} \mu(x,t^\prime) dt^\prime \]
which is strictly increasing and surjective onto $\mathbb{R}$ for fixed $x$, hence continuously invertible. Using (\ref{funct-tau-eq}) we see that $\gamma_\delta(x, t+t^\prime) = \gamma_\delta(\Phi_{\widehat{\mathcal{F}}}(x,t),t^\prime)$ and together with (\ref{funct-mu-eq}) this implies
\begin{equation} \label{funct-sig-eq}
\sigma(x,t+t^\prime) = \sigma(\Phi_{\widehat{\mathcal{F}}}(x,t), t^\prime) + \tau(x,t). 
\end{equation}

\vspace{0.2cm}

\noindent \textit{Claim 3.} The map $g: S^{2n-1}_{\epsilon^\prime} \rightarrow S^{2n-1}_{\epsilon}, x \mapsto \Phi_{\widehat{\mathcal{G}}}(h(x), \sigma(x,0))$ defines a homeomorphism inducing a topological equivalence of $\widehat{\mathcal{F}} \cap S^{2n-1}_{\epsilon^\prime}$ and $\widehat{\mathcal{G}} \cap S^{2n-1}_{\epsilon}$. 
\begin{proof}
If $\Phi_{\widehat{\mathcal{G}}}(h(x), \sigma(x,0)) = \Phi_{\widehat{\mathcal{G}}}(h(y), \sigma(y,0))$ then $h(y)$ is in the same $\widehat{\mathcal{G}}$-leaf as $h(x)$, hence $y$ is in the same $\widehat{\mathcal{F}}$-leaf as $x$, hence there is $t \in \mathbb{R}$ such that $y = \Phi_{\widehat{\mathcal{F}}}(x,t)$. Using (\ref{def-tau-eq}), (\ref{funct-sig-eq}) and the functorial properties of the flow $\Phi_{\widehat{\mathcal{G}}}$ we calculate
\begin{eqnarray*}
\Phi_{\widehat{\mathcal{G}}}(h(x), \sigma(x,0)) & = & \Phi_{\widehat{\mathcal{G}}}(h(y), \sigma(y,0)) = \Phi_{\widehat{\mathcal{G}}}(h(\Phi_{\widehat{\mathcal{F}}}(x,t)), \sigma(\Phi_{\widehat{\mathcal{F}}}(x,t),0)) = \\
 & = & \Phi_{\widehat{\mathcal{G}}}(\Phi_{\widehat{\mathcal{G}}}(h(x), \tau(x,t)), \sigma(\Phi_{\widehat{\mathcal{F}}}(x,t),0)) = \\
 & = & \Phi_{\widehat{\mathcal{G}}}(h(x), \tau(x,t) + \sigma(\Phi_{\widehat{\mathcal{F}}}(x,t),0)) = \\
 & = & \Phi_{\widehat{\mathcal{G}}}(h(x), \sigma(x,t)).
\end{eqnarray*}
If $\Phi_{\widehat{\mathcal{G}}|\{h(x)\} \times \mathbb{R}}$ is bijective, this implies $\sigma(x,0) = \sigma(x,t)$, hence $t=0$ by injectivity of $\sigma$ for fixed $x$, hence $y = \Phi_{\widehat{\mathcal{F}}}(x,0) = x$. If $\Phi_{\widehat{\mathcal{G}}|\{h(x)\} \times \mathbb{R}}$ is periodic with period $T_{\widehat{\mathcal{G}}, h(x)}$ and hence $\Phi_{\widehat{\mathcal{F}}|\{x\} \times \mathbb{R}}$ is periodic with period $T_{\widehat{\mathcal{F}}, x}$ then for some $k \in \mathbb{Z}$ we have 
\[ \sigma(x,t) = \sigma(x,0) + k \cdot T_{\widehat{\mathcal{G}}, h(x)} = \sigma(x,0) + k \cdot \tau(x, T_{\widehat{\mathcal{F}}, x}) = \sigma(x, k \cdot T_{\widehat{\mathcal{F}}, x}), \]
by (\ref{per-tau-eq}) and (\ref{funct-sig-eq}). Injectivity of $\sigma$ implies $t = k \cdot T_{\widehat{\mathcal{F}}, x}$, hence $y = \Phi_{\widehat{\mathcal{F}}}(x, k \cdot T_{\widehat{\mathcal{F}}, x}) = x$. So $g$ is injective. 

\noindent If $y \in S^{2n-1}_{\epsilon}$ then there exists $x \in S^{2n-1}_{\epsilon^\prime}$ such that $y = h(x)$, since $h$ is surjective. Then $y = \Phi_{\widehat{\mathcal{F}}}(h(x),0)$. Since $\sigma(x, \cdot)$ is surjective onto $\mathbb{R}$ there exists $t \in \mathbb{R}$
such that 
\begin{eqnarray*}
y = \Phi_{\widehat{\mathcal{G}}}(h(x), \sigma(x,t)) & = & \Phi_{\widehat{\mathcal{G}}}(h(x), \tau(x,t) + \sigma(\Phi_{\widehat{\mathcal{F}}}(x,t),0)) = \\
 & = & \Phi_{\widehat{\mathcal{G}}}(\Phi_{\widehat{\mathcal{G}}}(h(x), \tau(x,t)), \sigma(\Phi_{\widehat{\mathcal{F}}}(x,t),0)) = \\
 & = & \Phi_{\widehat{\mathcal{G}}}(h(\Phi_{\widehat{\mathcal{F}}}(x, t)), \sigma(\Phi_{\widehat{\mathcal{F}}}(x,t),0)) = \\
 & = & g(\Phi_{\widehat{\mathcal{F}}}(x,t)),
\end{eqnarray*} 
by (\ref{funct-sig-eq}), (\ref{def-tau-eq}) and the functorial property of the flow $\Phi_{\widehat{\mathcal{G}}}$. Hence $g$ is surjective.

\noindent As a bijective continuous map from a compact topological space to a Hausdorff space, $g$ is a homeomorphism. $g$ is also mapping leaves of $\widehat{\mathcal{F}} \cap S^{2n-1}_{\epsilon^\prime}$ to leaves of $\widehat{\mathcal{G}} \cap S^{2n-1}_{\epsilon}$, so $g$ is a topological equivalence of $\widehat{\mathcal{F}} \cap S^{2n-1}_{\epsilon^\prime}$ and $\widehat{\mathcal{G}} \cap S^{2n-1}_{\epsilon}$.
\end{proof}

\noindent This finishes the proof of the theorem.
\end{proof}

\section{The case of $\mathbb{R}$-linearly independent eigenvalues in dimension $2$} \label{nonreal-sec}

\noindent In this section, we only consider holomorphic foliation germs $\mathcal{F}$ around $0 \in \mathbb{C}^2$ represented by vector fields of the form 
\[ \lambda x \frac{\partial}{\partial x} + y\frac{\partial}{\partial y},\ \lambda \in \mathbb{C} - \mathbb{R}. \]
These foliation germs are invariant under the maps $\mathbb{C}^2 \rightarrow \mathbb{C}^2, (x,y) \mapsto r(x,y)$ for all $r \in \mathbb{R}_{>0}$. Hence there is a real intersection foliation $\mathcal{F} \cap S^3_{\epsilon}$ for all $\epsilon \in \mathbb{R}_{>0}$ as in section~\ref{intsectfol-sec}, and we assume from now on $\epsilon = 1$.

\begin{lem} \label{S1xS1-inv-lem}
Let $S^1 \times S^1$ act on $S^3$ by $(x,y) \mapsto (xe^{it_1}, ye^{it_2})$. Then the intersection foliation $\mathcal{F} \cap S^3$ is invariant under this action.
\end{lem}
\begin{proof}
The $1$-form $ydx - \lambda x dy$ corresponding to $\lambda x \frac{\partial}{\partial x} + y\frac{\partial}{\partial y}$ is pulled back to 
\[ ye^{it_2}d(xe^{it_1}) - \lambda xe^{it_1}d(ye^{it_2}) = e^{i(t_1+t_2)}(ydx - \lambda xdy) \]
by the action of $S^1 \times S^1$. Hence the tangent directions of the intersection foliation $\mathcal{F} \cap S^3$ are not changed, and the foliation is invariant under the action. 
\end{proof}

\noindent For $0 < \epsilon_x, \epsilon_y < 1$, denote the torus $\{(x,y) \in S^3: |x| = \epsilon_x\}$ by $T^x_{\epsilon_x}$ and the torus  $\{(x,y) \in S^3: |y| = \epsilon_y\}$ by $T^y_{\epsilon_y}$. Then $T^x_{\epsilon_x} = T^y_{\sqrt{1-\epsilon_x^2}}$.

\begin{lem} \label{torus-transv-lem}
$T^x_{\epsilon_x}$ intersects all the leaves of the intersection foliation $\mathcal{F} \cap S^3$  not lying on the coordinate axes exactly once and transversally.
\end{lem}
\begin{proof}
The real tangent vectors to the torus $T^x_{\epsilon_x}$ in a point $(x,y)$ are those real tangent vectors that are annihilated by the real differential forms
\[ d(x\overline{x}) = \overline{x}dx + xd\overline{x}\ \mathrm{and}\ d(y\overline{y}) = \overline{y}dy + yd\overline{y}. \]
The real tangent vectors to the leaf $L_{(x,y)}$ through $(x,y) \in T^x_{\epsilon_x}$ in $(x,y)$ are the $\mathbb{R}$-linear combinations of the real and imaginary part of $\lambda x \frac{\partial}{\partial x} + y \frac{\partial}{\partial y}$. Since
\[ (\overline{y}dy + yd\overline{y})(\lambda x \frac{\partial}{\partial x} + y\frac{\partial}{\partial y}) = \overline{y}y \in \mathbb{R}\ \mathrm{and}\ y \neq 0 \]
only the imaginary part of $\lambda x \frac{\partial}{\partial x} + y\frac{\partial}{\partial y}$ can be tangent to $T^x_{\epsilon_x}$. Since
\[ (\overline{x}dx + xd\overline{x})(\lambda x \frac{\partial}{\partial x} + y\frac{\partial}{\partial y}) = \lambda \overline{x}x\  \mathrm{and}\ x \neq 0 \]
this can only be the case if $\mathrm{Im}(\lambda) = 0$. But we assumed $\lambda \in \mathbb{C} - \mathbb{R}$, hence the real tangent spaces of $T^x_{\epsilon_x}$ and $L_{(x,y)}$ intersect transversally, hence the leaf  $L_{(x,y)} \cap S^3$ of $\mathcal{F} \cap S^3$  and $T^x_{\epsilon_x}$ intersect transversally.

\noindent In particular, on a leaf of $\mathcal{F} \cap S^3$ different from $\{ x = 0\}$ and $\{y = 0\}$ the absolute value of the $x$-coordinate must always strictly increase or decrease. Consequently, such a leaf intersects $T^x_{\epsilon_x}$ exactly once.
\end{proof}

\noindent For $0 < \epsilon_x, \epsilon_y < 1$  and $t \in \mathbb{R}$ denote the disk $\{ (x,\sqrt{1-|x|^2}e^{it}) \in S^3: |x| < \epsilon_x \}$ by $D^x_{t,\epsilon_x}$ and the disk  $\{ (\sqrt{1-|y|^2}e^{it},y) \in S^3: |y| < \epsilon_y \}$ by $D^y_{t,\epsilon_y}$.
\begin{lem} \label{disk-transv-lem}
$D^x_{t,\epsilon_x}$ and $D^y_{t,\epsilon_y}$ intersect all the leaves of the intersection foliation $\mathcal{F} \cap S^3$ everywhere transversally.
\end{lem}
\begin{proof}
By the $S^1 \times S^1$-invariance of the leaves of $\mathcal{F} \cap S^3$ shown in Lemma~\ref{S1xS1-inv-lem} we can assume that $t=0$. Since $D^x_{0,\epsilon_x}$ is an open subset of $\{y^2 = 1 - |x|^2\} \subset \mathbb{C}^2$, a smooth manifold for $|x| < 1$, the real tangent vectors to $D^x_{0,\epsilon_x}$ are exactly those annihilated by the real and the imaginary part of the differential form
\[ \omega = d(y^2 + |x|^2 - 1) = 2ydy + \overline{x}dx + xd\overline{x}. \]
We have
\[ \omega_{\mathrm{Re}} = ydy + \overline{y}d\overline{y} + \overline{x}dx + xd\overline{x}\ \mathrm{and}\ 
    \omega_{\mathrm{Im}} = -i(ydy - \overline{y}d\overline{y}). \]
Let $\theta(x,y)$ denote the complex tangent vector $\lambda x \frac{\partial}{\partial x} + y\frac{\partial}{\partial y}$ to the leaf $L_{(x,y)}$ through $(x,y) \in D^x_{0,\epsilon_x}$. Then $\omega_{\mathrm{Im}}(\theta(x,y)) = -iy^2 \in i\mathbb{R} - \{0\}$
since $y \in \mathbb{R} - \{0\}$. But the real part of $\theta(x,y)$ is not tangent to both $\{y^2 = 1 - |x|^2\}$ and $S^3 = \{ \overline{x}x + \overline{y}y = 1 \}$ either:
\[ \omega_{\mathrm{Re}}(\theta(x,y)) = 2y^2 + \lambda x \overline{x}\ \mathrm{and}\ d(\overline{x}x + \overline{y}y)(\theta(x,y)) = \lambda x \overline{x} + y\overline{y}, \]
hence the real part of the first number vanishes for $\mathrm{Re} \lambda = -\frac{2y^2}{|x|^2}$, the second for $\mathrm{Re} \lambda = -\frac{y\overline{y}}{|x|^2}$. Since $y \neq 0$ this cannot happen for the same $\lambda$. 
\end{proof}

\begin{figure}
	\centering
\includegraphics[width=0.8\textwidth]{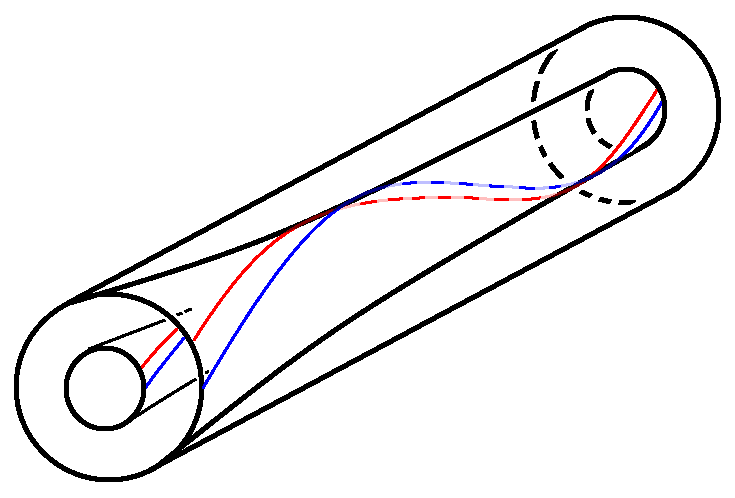}
  \caption{}
	\label{fig1}
\end{figure}

\noindent Figure~\ref{fig1} visualizes the behaviour of leaves of $\mathcal{F} \cap S^3$ in the cut-up solid torus $\bigcup_{0 \leq \epsilon_x \leq \epsilon} T^x_{\epsilon_x}$ (resp. $\bigcup_{0 \leq \epsilon_y \leq \epsilon} T^y_{\epsilon_y}$) as decribed by Lem.~\ref{torus-transv-lem} and~\ref{disk-transv-lem}.

\begin{thm} \label{R-LI-thm}
Let $\lambda_1 x \frac{\partial}{\partial x} + y\frac{\partial}{\partial y}$ and $\lambda_2 x \frac{\partial}{\partial x} + y\frac{\partial}{\partial y}$ represent two holomorphic foliation germs $\mathcal{F}_1, \mathcal{F}_2$ in $0 \in \mathbb{C}^2$, with $\lambda_1, \lambda_2 \in \mathbb{C} - \mathbb{R}$. Then $\mathcal{F}_1$ and $\mathcal{F}_2$ are topologically equivalent.
\end{thm}
\begin{proof}
We will construct a topological equivalence of the intersection foliations $\mathcal{F}_1 \cap S^3$ and $\mathcal{F}_2 \cap S^3$. Then the statement follows by the Reconstruction Theorem~\ref{topeq-intersec-fol-thm}.

\noindent Lemma~\ref{torus-transv-lem} and~\ref{disk-transv-lem} show that every leaf of $\mathcal{F}_i$ in the tubular torus $\{ (x,y) \in S^3: 0 < |x| \leq \frac{1}{2} \}$ is parametrized on the one hand by the absolute value $\epsilon_x$ of the $x$-coordinate, on the other hand by the argument $t$ of the $y$-coordinate. The parametrisation by $\epsilon_x$ yields the homeomorphisms
\[ \Phi_x^{(i)}: T_{1/2} \times (0, 1/2] \rightarrow \{ (x,y) \in S^3: 0 < |x| \leq 1/2 \} \] 
mapping a pair $(x,y) \times \epsilon_x$ to the unique intersection point of the leaf of $\mathcal{F}_i \cap S^3$ through $(x,y)$ with $T^x_{\epsilon_x}$.

\noindent Then $\Phi_x^{(2)} \circ (\Phi_x^{(1)})^{-1}$ is a homeomorphism of the  tubular torus $\{ (x,y) \in S^3: 0 < |x| \leq \frac{1}{2} \}$ into itself but might not be extendable to a homeomorphism of the solid torus $\{ (x,y) \in S^3: 0 \leq |x| \leq \frac{1}{2} \}$. To achieve that we reparametrize the $\epsilon_x$-interval $(0,\frac{1}{2}]$ using the second parametrization by the argument of the $y$-coordinate: Every leaf $L_{(x,y)}$ through a point $(x,y) \in T^x_{\frac{1}{2}}$ defines an invertible function $\phi^{(i)}_x: (0,\frac{1}{2}] \rightarrow [0, \infty)$, mapping $\epsilon_x$ to $t - t_0$ where $t$ is the argument of the $y$-coordinate of the intersection point of $L_{(x,y)}$ with $T^x_{\epsilon_x}$ and $t_0$ is the argument of $y$. These functions are the same for all such leaves because of the $S^1 \times S^1$-invariance, and we always have $\phi^{(i)}_x(\frac{1}{2}) = 0$. Then
\[ \Phi_x^{(2)} \circ \left[ \mathrm{id}_{T^x{\frac{1}{2}}} \times ((\phi^{(2)}_x)^{-1} \circ \phi^{(1)}_x) \right] \circ (\Phi_x^{(1)})^{-1} \]
maps $D^x_{t,\frac{1}{2}}$ and $T^x_{\epsilon_x}$, $0 < \epsilon_x \leq \frac{1}{2}$ onto themselves. This implies that the identity map on $\{x=0\} \cap S^3$ extends this composition of maps to a homeomorphism $\Phi_x$ of the solid torus $\{ (x,y) \in S^3: 0 < |x| \leq \frac{1}{2} \}$ mapping leaves of $\mathcal{F}_1 \cap S^3$ to leaves of $\mathcal{F}_2 \cap S^3$. Furthermore, the restriction of $\Phi_x$ to $T^x_{\frac{1}{2}}$ is the identity map.

\noindent In the same way we can construct a homeomorphism $\Phi_y$ of the solid torus $\{ (x,y) \in S^3: 0 < |y| \leq \frac{1}{2} \}$ mapping leaves of $\mathcal{F}_1 \cap S^3$ to leaves of $\mathcal{F}_2 \cap S^3$. Since again the restriction of $\Phi_y$ to $T^x_{\frac{1}{2}}$ is the identity map $\Phi_x$ and $\Phi_y$ glue to a topological equivalence of the intersection foliations $\mathcal{F}_1 \cap S^3$ and $\mathcal{F}_2 \cap S^3$ .
\end{proof}

\begin{rem}
The theorem is Guckenheimer's result in dimension $2$ \cite{Guc72}. The proof above yields the construction of an explicit topological equivalence which is missing in Guckenheimer's original argument. Another explicit topological equivalence is constructed in \cite{CKP78} using polycylinders instead of balls.
\end{rem}

\section{The resonant case in dimension 2} \label{resonant-sec}

\noindent In this section, we only consider holomorphic foliation germs $\mathcal{F}_m$ around $0 \in \mathbb{C}^2$ represented by vector fields $\theta_m$ of the form 
\[  (mx + y^m) \frac{\partial}{\partial x} + y\frac{\partial}{\partial y},\ m \geq 1. \] 
Note that all these foliation germs are equal to the germs in Rem.~\ref{dim2-fol-rem}.(3) and (4), up to possibly rescaling the $x$-coordinate.

\begin{lem}
The leaves of $\mathcal{F}_m$ intersect all spheres $S^3_{\epsilon}$, $0 < \epsilon \leq 1$, transversally. In particular the real intersection foliation $\mathcal{F}_m \cap S^3$ on $S^3 = S^3_1$ exists.
\end{lem}
\begin{proof}
The leaf of $\mathcal{F}_m$ through $p=(x,y)$ does not intersect $S^3_{\epsilon}$ transversally in $p$ if, and only if the holomorphic tangent vector $\theta_m(p)$ is tangent to $S^3_{\epsilon}$ in $p$ if, and only if
\[ \left( \overline{x}dx + \overline{y}dy + x d\overline{x} + y d\overline{y} \right) 
   \left( (mx + y^m) \frac{\partial}{\partial x} + y\frac{\partial}{\partial y} \right) = \overline{x} x + \overline{y} y + (m-1)  \overline{x} x + \overline{x} y^m = 0. \] 
But this is impossible since $ \overline{x} x +  \overline{y} y = \epsilon^2$, $(m-1)  \overline{x} x \geq 0$ and 
\[  |\overline{x} y^m| =  |x| \cdot |y|^m < \epsilon^{m+1} \leq \epsilon^2, \]
since $|x|, |y| < \epsilon$ but never $|x| = |y| = \epsilon$.
\end{proof}

\noindent Next, we analyse the leaves of the intersection foliations $\mathcal{F}_m \cap S^3$.
\begin{prop} \label{leaf-res-prop}
The only closed leaf of $\mathcal{F}_m \cap S^3$ is $\{y = 0\} \cap S^3$. The closure of any leaf $L_{(a,b)}$ through a point $(a,b) \in S^3 - \{y = 0\}$ is $L_{(a,b)} \cup (\{y = 0\} \cap S^3)$. For a certain $\epsilon_y = \epsilon_y(L_{(a,b)})$ with $0 < \epsilon_y \leq 1$ the leaf $L_{(a,b)}$ intersects a torus $T^y_{\epsilon_y^\prime}$ 
\begin{itemize}
\item in two distinct points if $0 < \epsilon_y^\prime < \epsilon_y$, 
\item in one point if $\epsilon_y^ \prime = \epsilon_y$ and 
\item not at all if $\epsilon_y^\prime > \epsilon_y$. 
\end{itemize}
\end{prop} 
\begin{proof}
The holomorphic map 
\[ \lambda_{(a,b)}: \mathbb{C} \rightarrow \mathbb{C}^2,\ t \mapsto ((a + b^mt)e^{mt}, be^t) \]
defines the integral curve of the vector field $(mx + y^m) \frac{\partial}{\partial x} + y\frac{\partial}{\partial y}$ through $\lambda_{(a,b)}(0) = (a,b)$, that is the leaf of $\mathcal{F}_m$ through $(a,b)$. If $(a,b) \in S^3$ the leaf of $\mathcal{F}_m \cap S^3$ through $(a,b)$ is the $\lambda_{(a,b)}$-image of the branch through $t=0$ of the curve in $\mathbb{C}$ implicitely given by
\[ 1 = e^{m(t + \overline{t})}(a\overline{a} + b^m\overline{a}t + a\overline{b}^m\overline{t} + (b\overline{b})^mt\overline{t}) + b\overline{b}e^{t + \overline{t}}. \]
Decomposing $t = t_R + it_I$ into real and imaginary part and rearranging the equation we obtain
\[ \tag{$\ast$} (b\overline{b})^m t_I^2 + 2 \mathrm{Im}(a\overline{b}^m) t_I + a\overline{a} + 2 \mathrm{Re}( b^m\overline{a})
    t_R + (b\overline{b})^m t_R^2 +  b\overline{b}e^{2(1-m)t_R} - e^{-2mt_R} = 0. \]
This is a quadratic equation in $t_I$, with coefficients of $t_I^2$ and $t_I$ independent of $t_R$.

\vspace{0.2cm}

\noindent \textit{Claim.} For $t_R \leq 0$ the constant term of ($\ast$) is increasing with $t_R$.
\begin{proof}
When we derive the constant term with respect to $t_R$ we obtain the gradient
\[ 2 \mathrm{Re}( b^m\overline{a}) + 2 (b\overline{b})^m t_R + 2(1-m)b\overline{b}e^{2(1-m)t_R} + 2me^{-2mt_R} \]
which is $> 2e^{-2mt_R} + 2t_R + 2(1-m)e^{2t_R} - 2$ for $t_R \leq 0$ since $(a,b) \in S^3$ implies $|a|, |b| \leq 1$ and $|\mathrm{Re}( b^m\overline{a})| < 1$. But the function $x \mapsto e^{-2mx} + x + (1-m)e^{2x} - 1$ has derivative $-2me^{-2mx} + 2(1-m)e^{2x} + 1 < 0$ for $x \leq 0$, hence the gradient is always $> 0$ for $t_R \leq 0$, as it is $> 0$ for $t_R = 0$.
\end{proof}

\noindent If $t_R \rightarrow \infty$ the constant term also tends to $\infty$. So we conclude: There exists a $t_0 \geq 0$ such that for all $t_R < t_0$ there are two solutions $t_I$ to the equation ($\ast$) symmetric to $-\frac{\mathrm{Im}(a\overline{b}^m)}{(b\overline{b})^m}$, and one solution $t_I = -\frac{\mathrm{Im}(a\overline{b}^m)}{(b\overline{b})^m}$ if $t_R = t_0$.

\noindent In particular, if $t_R \rightarrow -\infty$ we have $y = be^t \rightarrow 0$ which implies the claim on the closure of $L_{(a,b)}$. Since this leaf intersects a torus $T^y_{\epsilon_y^\prime}$ in all points $(a^\prime, b^\prime)$ of the leaf where $|b^\prime| = \epsilon_y^\prime$ the last claim follows. In particular, the maximal $\epsilon_y$ such that $L_{(a,b)} \cap T^y_{\epsilon_y} \neq \emptyset$ is given by $\epsilon_y = |e^{t_0}b|$.
\end{proof}

\begin{cor} \label{leaf-par-cor}
All the leaves of $\mathcal{F}_m \cap S^3$ away from $\{y = 0\}$ are uniquely parametrised by the points of the set
\[ \{ (a,b) \in S^3: \mathrm{Im}(a\overline{b}^m) = 0, b \neq 0 \}. \]
\end{cor}
\begin{proof}
Since there is only one point on a leaf $L_{(a,b)}$ with maximal distance $\epsilon_y(L_{(a,b)})$ to $\{y = 0\}$, these points uniquely parametrise all leaves of $\mathcal{F}_m$ away from $\{y = 0\}$. Furthermore, $(a,b)$ is such a point on $L_{(a,b)}$ if for $t=0$ the linear and constant term of ($\ast$) vanish. This is exactly the case when $\mathrm{Im}(a\overline{b}^m) = 0$ since $a\overline{a} + b\overline{b} = 1$.
\end{proof}

\begin{thm}
The intersection foliations $\mathcal{F}_m \cap S^3$ are not topologically equivalent for different $m = 1, 2, \ldots$.
\end{thm}
\begin{proof}
Assume that $\Phi: S^3 \rightarrow S^3$ is a topological equivalence of $\mathcal{F}_{m_1} \cap S^3$ with $\mathcal{F}_{m_2} \cap S^3$. Then $\Phi$ maps the only closed leaf of $\mathcal{F}_{m_1}$ to the only closed leaf of $\mathcal{F}_{m_2}$, that is, $\{y = 0\} \cap S^3$ to itself. Hence $\Phi$ maps the open complement $U_1 := S^3 - \bigcup_{0 \leq \epsilon_y \leq \epsilon_1} T^y_{\epsilon_y}$ of the solid torus $\bigcup_{0 \leq \epsilon_y \leq \epsilon_1} T^y_{\epsilon_y}$ to an open set $\Phi(U_1)$ in $S^3$ not intersecting $\{y = 0\} \cap S^3$ but  containing $\{x = 0\} \cap S^3$ if $\epsilon_1$ is small enough, by a compactness argument.

\noindent Let $\overline{U_1}$ be the union of all leaves of  $\mathcal{F}_{m_1} \cap S^3$  intersecting $U_1$. Then the complement $V_1 := S^3 - (\overline{U_1} \cup \{y = 0\})$ consists of leaves of the foliation $\mathcal{F}_{m_1} \cap S^3$. Cor.~\ref{leaf-par-cor} shows that these leaves are uniquely parametrised by points of $\{\mathrm{Im}(a\overline{b}^{m_1}) = 0\} \cap \bigcup_{0 < \epsilon_y \leq \epsilon_1} T^y_{\epsilon_y}$. 

\noindent Note that for $0 < \epsilon_y < 1$ the intersection  $\{\mathrm{Im}(a\overline{b}^m) = 0\} \cap T^y_{\epsilon_y}$ consists of $m$ connected curves given by $m\mathrm{arg}(b) - \mathrm{arg}(a) \in \pi \cdot \mathbb{Z}$ on the torus $T^y_{\epsilon_y}$, each of them of homology class $(m,1)$ with respect to the generating cycles $\{\mathrm{arg}(x) = 0\} \cap T^y_{\epsilon_y}$ and $\{\mathrm{arg}(y) = 0\} \cap T^y_{\epsilon_y}$. These curves are visualized in Figure~\ref{fig2} when $m=2$, as the red and the blue curve on the torus $T^y_{\epsilon_y}$ cut up along a disk $D^y_t$. Hence 
$\{\mathrm{Im}(a\overline{b}^m) = 0\} \cap \bigcup_{0 < \epsilon_y \leq \epsilon_1} T^y_{\epsilon_y}$ has $m$ connected components, and all of them can be retracted to a curve of homology class $(m,1)$ in $T^y_{\epsilon_1}$. Since $S^3 - \{y=0\}$ can be retracted to $S^3 \cap \{x=0\}$, the homology class of this curve in $S^3 - \{y=0\}$ is $m$ times the generator represented by $S^3 \cap \{x=0\}$.

\noindent The flow on $S^3$ associated to $\mathcal{F}_{m_1}$ induces a retraction of $V_1$ to $\{\mathrm{Im}(a\overline{b}^{m_1}) = 0\} \cap \bigcup_{0 < \epsilon_y \leq \epsilon_1} T^y_{\epsilon_y}$, hence $V_1$ consists of $m_1$ connected components $V_1^\prime, \ldots, V_1^{(m_1)}$. These components are visualized in Figure~\ref{fig2} when $m_1=2$, as the two regions enclosed by the red and the blue surfaces in the cut-up solid torus $\bigcup_{0 \leq \epsilon_y \leq \epsilon_1} T^y_{\epsilon_y}$. By construction, $\Phi(V_1)$ does not intersect the complement of a solid torus, $U_2 := S^3 - \bigcup_{0 \leq \epsilon_y \leq \epsilon_2} T^y_{\epsilon_y}$ if $\epsilon_2$ is close enough to $1$. Constructing $V_2$ from $U_2$ using $\mathcal{F}_{m_2}$ as $V_1$ was constructed from $U_1$ using $\mathcal{F}_{m_1}$, this implies $\Phi(V_1) \subset V_2$, and $\Phi(V_1^\prime)$ lies in one of the $m_2$ connected components of $V_2$, say $V_2^\prime$.

\noindent Consequently, we have a commutative diagram of homeomorphisms and embeddings,
\[ \begin{array}{ccccc}
    V_1^\prime & \stackrel{\Phi}{\rightarrow} & \Phi(V_1^\prime) & \subset & V_2^\prime \\
    \cap & & \cap &  & \cap \\
    S^3 - \{y=0\} & \stackrel{\Phi}{\rightarrow} & S^3 - \{y=0\} & =  & S^3 - \{y=0\}. \\
    \end{array} \]
This diagram induces the commutative diagram of group homomorphisms of homology group
\[ \begin{array}{ccccc}
    \rule{.5cm}{0cm} \mathbb{Z}  & \stackrel{\cdot \pm 1}{\longrightarrow} & \mathbb{Z} & \rightarrow & \mathbb{Z} \rule{.5cm}{0cm} \\
    \cdot m_1 \downarrow & & & & \downarrow \cdot m_2 \\
    \rule{.5cm}{0cm} \mathbb{Z} & = & \mathbb{Z} & = & \mathbb{Z} \rule{.5cm}{0cm}
    \end{array} \]
The left and right vertical homomorphism are given by multiplications with $m_1$ and $m_2$ because of the retractions constructed above, whereas the upper right homomorphism is given by multiplication with an arbitrary integer $n$.

\noindent Consequently, we obtain $\pm m_1 = \pm n \cdot m_2$, hence $m_1 \geq m_2$. Exchanging the roles of $m_1$ and $m_2$ we also obtain $m_1 \leq m_2$ and therefore $m_1 = m_2$.
\end{proof}

\begin{figure}
	\centering
    \includegraphics[width=0.8\textwidth]{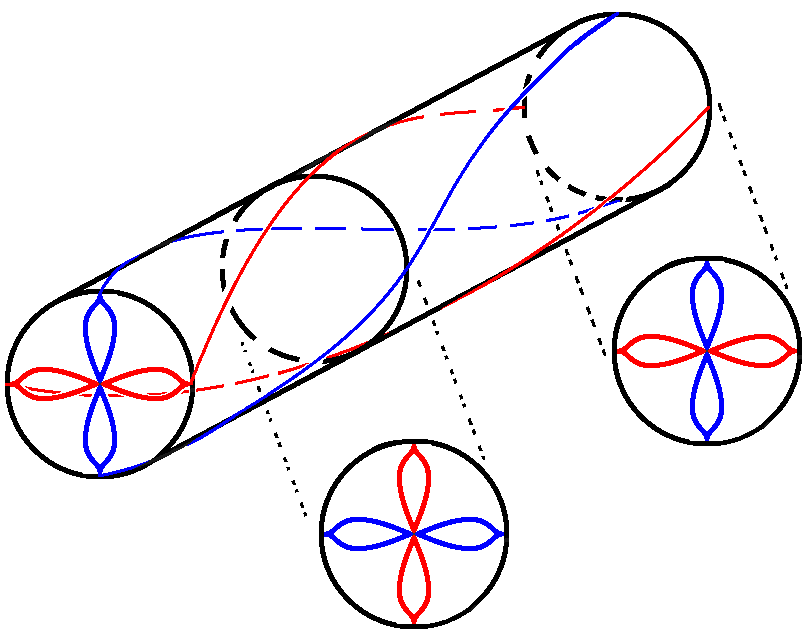}
  \caption{}
	\label{fig2}
\end{figure}

\begin{rem} \label{not-GT-rem}
The holomorphic foliation germs $\mathcal{F}_m$ discussed in this section are not of general type, in the terminology of \cite{MaMa12}: One feature of plane holomorphic foliation germs of general type is that the singularities of the reduction are represented by vector fields without a linear part with eigenvalue $0$. But from $(mx + y^m) \frac{\partial}{\partial x} + y\frac{\partial}{\partial y}$ respectively the holomorphic $1$-form $y dx - (mx + y^m) dy$ representing the same holomorphic foliation germ we obtain $1$-forms resp.\ vector fields
\[ t(1-m+t^mx^{m-1})dx - (mx+t^mx^m)dt\ \mathrm{resp.\ } 
   (mx+t^mx^m) \frac{\partial}{\partial x} + t(1-m+t^mx^{m-1}) \frac{\partial}{\partial t} \]
in the $(x,t)$-chart with $(x,y) = (x,xt)$ and
\[ y ds + (s(1-m) + y^{m-1}) dy\ \mathrm{resp.\ } y \frac{\partial}{\partial y} + ((m-1)s + y^{m-1}) \frac{\partial}{\partial s} \]
in the $(s,y)$-chart with $(x,y) = (sy,y)$, by blowing up $\mathbb{C}^2$ in $0$.

\noindent If $m = 1$ the blown-up foliation in the $(x,t)$-chart is represented by $x \frac{\partial}{\partial x} + t^2 \frac{\partial}{\partial t}$, yielding a reduced singularity in $(x,t) = (0,0)$ but not one of general type. 

\noindent If $m \geq 2$ the blown-up foliation has a singularity of type $\mathcal{F}_{m-1}$ in $(s,y) = (0,0)$. Thus further reducing this singularity will finally lead to another reduced singularity not of general type. 
\end{rem}

\section{The non-resonant case of $\mathbb{R}$-linearly dependent eigenvalues in dimension $2$}

\noindent In this section, we only consider holomorphic foliation germs $\mathcal{F}_{\lambda}$ around $0 \in \mathbb{C}^2$ represented by vector fields of the form 
\[  \lambda x \frac{\partial}{\partial x} + y\frac{\partial}{\partial y},\ \lambda \in \mathbb{R}_{>0}. \]
As in Section~\ref{nonreal-sec} these foliation germs are invariant under rescaling with positive real constants. Hence it is enough to consider the real intersection foliations $\mathcal{F}_{\lambda} \cap S^3_1 = \mathcal{F}_{\lambda} \cap S^3$.
\begin{lem} \label{real-leaf-lem}
Every leaf of the intersection foliation $\mathcal{F}_\lambda \cap S^3$ lies on a torus $T^x_{\epsilon_x}$, $0 \leq \epsilon_x \leq 1$.
\end{lem}
\begin{proof}
The flow of the vector field $\lambda x \frac{\partial}{\partial x} + y\frac{\partial}{\partial y}$ is given by $(a,b,t) \mapsto (ae^{\lambda t}, be^{t})$. Since $\lambda \in \mathbb{R}_{>0}$ the intersection of the associated integral manifold through a point $(a,b) \in S^3$ with $S^3$ is parametrised by $t \mapsto (ae^{\lambda i t}, be^{it})$. Thus the leaf of $\mathcal{F}_\lambda \cap S^3$ through $(a,b)$ lies on the torus $T_x(|a|)$.
\end{proof}

\subsection{$\lambda \in \mathbb{Q}_{>0}$} \label{rational-ssec}

Assume that $\lambda = \frac{p}{q}$, where $p, q \in \mathbb{N}$ are relatively prime.
\begin{prop} \label{nonres-lin-leaf-prop}
Every leaf of the intersection foliation $\mathcal{F}_\lambda \cap S^3$ is closed. A leaf on the torus $T_x(\epsilon_x)$, $0 < \epsilon_x  < 1$, is a curve of type $(p,q)$, where $p$ describes the winding number of the leaf around $\{x = 0\} \cap S^3$ and $q$ the winding number around $\{y = 0\} \cap S^3$. The holonomy in a point $(0, e^{it}) \in \{x = 0\} \cap S^3$ following the leaf in counter-clockwise direction is given by the germ of the map $D^t_x(\epsilon_x) \rightarrow D^t_x(\epsilon_x)$, $0 < \epsilon_x \ll 1$, multiplying the $x$-coordinate by $e^{2\pi i \cdot \frac{p}{q}}$. Similarly, the holonomy of the leaf in a point $(e^{it}, 0)$ following the leaf in counter-clockwise direction is described by the germ of the map $D^t_y(\epsilon_y) \rightarrow D^t_y(\epsilon_y)$ multiplying the $y$-coordinate with $e^{2\pi i \cdot \frac{q}{p}}$.

\noindent The holonomy in all points of $S^3$ away from $\{x = 0\} \cup \{y = 0\}$ is the identity.
\end{prop}
\begin{proof}
$\mathcal{F}_{\lambda}$ is also represented by the vector field $p  x \frac{\partial}{\partial x} + q y \frac{\partial}{\partial y}$. The flow of this vector field is given by $(a,b,t) \mapsto (ae^{p t}, be^{q t})$, and the intersection of the associated integral manifold through $(a,b)$ with $S^3$ is parametrised as $t \mapsto (ae^{pit}, be^{qit})$, $t \in \mathbb{R}$. These parametrisations are periodic, with period $\frac{2\pi}{\gcd(p,q)} = 2\pi$. The claims of the proposition follow. 
\end{proof}

\begin{cor}
Two foliation germs $\mathcal{F}_{\lambda}$, $\mathcal{F}_{\mu}$, $\lambda, \mu \in \mathbb{Q}_{>0}$, are topologically equivalent if, and only if $\lambda = \mu$ or $= \frac{1}{\mu}$.
\end{cor}
\begin{proof}
By the Reconstruction Theorem~\ref{topeq-intersec-fol-thm} we only have to decide whether the intersection foliations $\mathcal{F}_{\lambda} \cap S^3$ and $\mathcal{F}_{\mu} \cap S^3$ are topologically equivalent or not. Now, the topological types of the holonomy along closed paths on leaves of these real foliation are topologically invariant, in particular the order of the holonomy germ. Consequently, Prop.~\ref{nonres-lin-leaf-prop} implies that only $\mathcal{F}_{\frac{p}{q}} \cap S^3$ and $\mathcal{F}_{\frac{q}{p}} \cap S^3$ can be topologically equivalent, and in that case the equivalence is given by $(x,y) \mapsto (y,x)$.
\end{proof}

\subsection{$\lambda \in \mathbb{R}_{>0} - \mathbb{Q}_{>0}$} \label{irrational-ssec}

As in the proof of Lemma~\ref{real-leaf-lem} the leaf of the intersection foliation $\mathcal{F}_\lambda \cap S^3$ through a point $(a,b) \in S^3$ is parametrised by $t \mapsto (ae^{i \lambda t}, be^{it})$, hence lies on $T_x(|a|)$. Since $\lambda$ is irrational the leaf is not closed but dense on the torus $T_x(|a|)$, for all $a \in \mathbb{C}$ such that $0 < |a| < 1$. Thus we can describe the leaves of $\mathcal{F}_\lambda \cap S^3$ as follows: 
\begin{lem} \label{irrat-leaf-lem}
The intersection foliation $\mathcal{F}_\lambda \cap S^3$ has two closed leaves, $\{x = 0\} \cap S^3$  and $\{y = 0\} \cap S^3$, whereas the closure of every other leaf is a torus $T_x(\epsilon_x)$, $0 < \epsilon_x < 1$. \hfill $\Box$
\end{lem}

\noindent Next, we consider the continuous map $f: S^3 \rightarrow [0,1], (x,y) \mapsto |x|$. Its fibers are $f^{-1}(\epsilon_x) = T_x(\epsilon_x)$, $0 \leq \epsilon_x \leq 1$. Lemma~\ref{irrat-leaf-lem} shows that a topological equivalence $\Phi$ of $\mathcal{F}_\lambda \cap S^3$ with $\mathcal{F}_\mu \cap S^3$, $\lambda, \mu \in \mathbb{R}_{>0} - \mathbb{Q}_{>0}$, induces a homeomorphism $\phi: [0,1] \rightarrow [0,1]$ such that $\phi \circ f = f \circ \Phi$, with $\phi(\{0,1\}) = \{0,1\}$. Note that $\phi(0) = 0$ and $\phi(1) = 1$ means that $\Phi$ maps the closed leaves $\{x = 0\} \cap S^3$  resp. $\{y = 0\} \cap S^3$ onto themselves, whereas $\phi(0) =1$, $\phi(1) = 0$ indicates that $\Phi$ interchanges the closed leaves.

\noindent Furthermore, $\Phi$ maps the torus $T_x(\epsilon_x)$ homeomorphically onto the torus $T_x(\phi(\epsilon_x))$, $0 < \epsilon_x < 1$. Recall that the (extended) mapping class group of a $2$-dimensional torus $T^2 \cong S^1 \times S^1$ is given by $GL(H_1(T^2),\mathbb{Z})$ \cite[Thm.2.5]{FM12}. Identifying the tori $T_x(\epsilon_x)$ for different $0 < \epsilon_x < 1$ by rescaling the $x$- and the $y$-coordinate the following statement makes sense:
\begin{prop} \label{map-class-prop}
If $\Phi: S^3 \rightarrow S^3$ is a topological equivalence of $\mathcal{F}_\lambda \cap S^3$ with $\mathcal{F}_\mu \cap S^3$, $\lambda, \mu \in \mathbb{R}_{>0} - \mathbb{Q}_{>0}$ then the restriction $\Phi_{|T_x(\epsilon_x)}: T_x(\epsilon_x) \rightarrow T_x(\phi(\epsilon_x))$ is of one of the types $\left( \begin{array}{cc} \pm 1 & 0 \\ 0 & \pm 1 \end{array} \right)$ or $\left( \begin{array}{cc} 0 & \pm 1 \\ \pm 1 & 0 \end{array} \right)$ in the mapping class group of a $2$-dimensional torus, for all $0 < \epsilon_x < 1$.
\end{prop} 
\begin{proof}
Interchanging the coordinates yields a homeomorphism $\Psi: S^3 \rightarrow S^3, (x,y) \mapsto (y,x)$ whose restriction to tori $T_x(\epsilon_x)$ is of type $\left( \begin{array}{cc} 0 & 1 \\ 1 & 0 \end{array} \right)$ in the mapping class group of a $2$-dimensional torus. Composing $\Psi$ with a topological equivalence $\Phi$ of $\mathcal{F}_\lambda \cap S^3$ with $\mathcal{F}_\mu \cap S^3$ such that $\phi(0) = 1$, $\phi(1) = 0$ yields a topological equivalence $\Phi^\prime$  of $\mathcal{F}_\lambda \cap S^3$ with $\mathcal{F}_{\frac{1}{\mu}} \cap S^3$ such that $\phi^\prime(0) = 0$, $\phi^\prime(1) = 1$. Hence, from now on we will only consider that case.

\noindent For all $0 < \epsilon_0 < 1$ the topological equivalence $\Phi$ maps the solid torus $\bigcup_{0 \leq \epsilon_x \leq \epsilon_0} T_x(\epsilon_x)$ homeomorphically onto the  solid torus $\bigcup_{0 \leq \epsilon_x \leq \epsilon_0} T_x(\phi(\epsilon_x))$
and  $\bigcup_{\epsilon_0 \leq \epsilon_x \leq 1} T_x(\epsilon_x)$ onto $\bigcup_{\epsilon_0 \leq \epsilon_x \leq 1} T_x(\phi(\epsilon_x))$, always mapping the tori $T_x(\epsilon_x)$ onto $T_x(\phi(\epsilon_x))$. The fundamental groups of these solid tori are generated by $L_x := \{x = 0\} \cap S^3$ resp. $L_y := \{y = 0\} \cap S^3$, and a curve of type $(p,q)$ on the torus $T_x(\epsilon_x)$ (for the notation, see Prop.~\ref{nonres-lin-leaf-prop}) is mapped to the class of $q \cdot L_x$ resp. $p \cdot L_y$ by the inclusion into the solid tori. Consequently, the homeomorphism $\Phi$ must map a curve of type $(p,q)$ on $T_x(\epsilon_x)$ to a curve of type $(\pm p, \pm q)$ on $T_x(\phi(\epsilon_x))$. This implies the claim on the isotopy classes of $\Phi_{|T_x(\epsilon_x)}$.
\end{proof}

\noindent To finally classify the holomorphic foliation germs $\mathcal{F}_\lambda$, $\lambda \in \mathbb{R}_{>0} - \mathbb{Q}_{>0}$, we consider Kronecker foliations $F_\lambda$, $\lambda \in \mathbb{R}_{>0}$, on the $2$-dimensional torus $T^2 = S^1 \times S^1$. These foliations are given by the orbits of the flow
\[ t \cdot_\lambda (e^{ia}, e^{ib}) = (e^{i(a+\lambda t)}, e^{i(b+t)}),\ t, a, b \in \mathbb{R}. \]

\begin{prop} \label{Kron-prop}
Two Kronecker foliations $F_\lambda$ and $F_\mu$, $\lambda, \mu \in \mathbb{R}$, are topologically equivalent if $\mu = \frac{a\lambda + b}{c\lambda + d}$, where $\left( \begin{array}{cc} a & b \\ c & d \end{array} \right) \in GL(2,\mathbb{Z})$.
\end{prop}
\begin{proof}
Let $Q := \left( \begin{array}{cc} a & b \\ c & d \end{array} \right) \in GL(2,\mathbb{Z})$. Then 
\[ \phi_Q: T^2 \rightarrow T^2, (e^{ix}, e^{iy}) \mapsto (e^{i(ax+by)}, e^{i(cx+ dy)}) \] 
is a homeomorphism with inverse map $\phi_{Q^{-1}}$. Since for $s = (c\lambda + d)t$,
\begin{eqnarray*}
\phi_Q(t \cdot_\lambda (e^{ix}, e^{iy})) & = & (e^{i(ax + by + (a\lambda + b)t)}, e^{i(cx + dy + (c\lambda + d)t)}) = \\
                                                                   & = & (e^{i(ax + by + \mu s)}, e^{i(cx + dy + s)}) = \\
                                                                   & = & s \cdot_\mu \phi_Q(e^{ix}, e^{iy}),
\end{eqnarray*}
$\phi_Q$ is a topological equivalence of $F_\lambda$ and $F_\mu$.
\end{proof}

\noindent If  $\lambda, \mu \in \mathbb{R}_{>0} - \mathbb{Q}_{>0}$ the converse is also true, as the following theorem shows: 
\begin{thm} \label{Kron-hom-thm}
Let $\phi: T^2 \rightarrow T^2$ be a topological equivalence of Kronecker foliations $F_\lambda$ and $F_\mu$, $\lambda, \mu \in \mathbb{R}_{>0} - \mathbb{Q}_{>0}$. If $\phi$ has the homotopy type $\left( \begin{array}{cc} a & b \\ c & d \end{array} \right)$ in the mapping class group $GL(2,\mathbb{Z})$ of $T^2$ then $\mu = \frac{a\lambda + b}{c\lambda + d}$.
\end{thm}
\begin{proof}
First of all, we may assume that $\left( \begin{array}{cc} a & b \\ c & d \end{array} \right) = \left( \begin{array}{cc} 1 & 0 \\ 0 & 1 \end{array} \right)$, that is, $\phi$ is isotopic to the identity: If not, Prop.~\ref{Kron-prop} shows that $F_\mu$ is topologically equivalent to $F_{Q^{-1} \cdot \mu}$ where $Q^{-1}$ is the inverse matrix of $Q = \left( \begin{array}{cc} a & b \\ c & d \end{array} \right)$. Furthermore, the topological equivalence $\phi_{Q^{-1}}$ is of homotopy type $Q^{-1} \in GL(2, \mathbb{Z})$, so the topological equivalence $\phi_{Q^{-1}} \circ \phi$ between $F_\lambda$ and $F_{Q^{-1} \cdot \mu}$ is of homotopy type $\left( \begin{array}{cc} 1 & 0 \\ 0 & 1 \end{array} \right)$. Consequently, if we show that $\lambda = Q^{-1} \cdot \mu$, then as claimed
\[ \mu = Q \cdot \lambda = \frac{a\lambda + b}{c\lambda + d}. \]

\noindent For a given $\lambda \in \mathbb{R}_{>0} - \mathbb{Q}_{>0}$ and a point $P = (e^{ia}, e^{ib}) \in T^2$, let
\[ L_P^{(\lambda)} := \{ (e^{i(a + \lambda t)}, e^{i(b + t)}) | t \in \mathbb{R} \} \subset T^2 \]
be the leaf of $F_{\lambda}$ through $P$. Following ideas from ergodic theory we express the "slope" of the leaf $L_P^{(\lambda)}$ as a quotient of its topological intersection numbers with two curves representing generators of $H_1(T^2, \mathbb{Z})$. To this purpose we need arbitrarily long pieces of the leaf $L_P^{(\lambda)}$ starting in $P$ and ending in $P^\prime$ arbitrarily close to $P$. Here, we measure distances on $T^2$ using the metric induced by the Euclidean metric on the universal covering $\mathbb{R}^2$. 

\noindent So consider the preimage $p^{-1}(B_P(\epsilon))$ of a ball $B_P(\epsilon)$, $0 < \epsilon \ll 1$ under the parametrisation $p: \mathbb{R} \rightarrow L_P^{(\lambda)} \subset T^2$ given by $t \mapsto (e^{i(a + \lambda t)}, e^{i(b + t)})$. Since $\lambda$ is irrational, $L_P^{(\lambda)}$ is dense in $T^2$, hence $p^{-1}(B_P(\epsilon))$ consists of infinitely many intervals in arbitrarily large distances to $0 \in \mathbb{R}$. One of the intervals in $p^{-1}(B_P(\epsilon))$, say $I_0$, contains $0$, whereas the images of all the other intervals have a non-zero distance to $P$. In particular, if $\epsilon \rightarrow 0$ then the boundaries of all the intervals not containing $0$ tend to $\pm \infty$. This observation holds for the intervals in the preimage of an arbitrary neighborhood basis of $P$. 

\noindent Let $I_1$ be the interval in $p^{-1}(B_P(\epsilon))$ closest to the right to $I_0$. As indicated in Figure~\ref{fig3} we can construct a closed path $\gamma_{P,\epsilon}^{(\lambda)}: [0,1] \rightarrow T^2$ starting and ending in $P$, by following the leaf $L_P^{(\lambda)}$ to a point $P^\prime \in p(I_1)$ and connecting $P^\prime$ to $P$ by a path inside $B_P(\epsilon)$.

\begin{figure}
	\centering

\begin{picture}(130,120)(-10,-10)
\put (0,0){\line(1,0){105}}
\put (0,10){\line(0,1){90}}
\put (0,100){\line(1,0){105}}
\put (105,0){\line(0,1){100}}

\put (45,0){\line(-2,1){10}}
\put (45,0){\line(-2,-1){10}}
\put (0,50){\line(-1,-2){5}}
\put (0,50){\line(1,-2){5}}
\put (65,100){\line(-2,1){10}}
\put (65,100){\line(-2,-1){10}}
\put (105,50){\line(-1,-2){5}}
\put (105,50){\line(1,-2){5}}

\qbezier(0,15)(15,15)(15,0)
\put (7,-2){\line(0,1){8}}
\put (2,-10){\small $B_P(\epsilon)$}

\put (90,80){\line(1,0){20}}
\put (111,75){\small $\gamma_{P,\epsilon}^{(\lambda)}$}

\put (0,10){\circle*{2}}
\put (-10,6){\small $P^\prime$}
\put (105,10){\circle*{2}}
\put (107,6){\small $P^\prime$}
\put (0,0){\circle*{2}}
\put (-10,-6){\small $P$}

\color{red}
\put (0,0){\line(1,2){50}}
\put (50,0){\line(1,2){50}}
\put (100,0){\line(1,2){5}}
\put (105,0){\line(0,1){10}}
\put (0,0){\line(0,1){10}}
\end{picture}
   
  \caption{}
	\label{fig3}
\end{figure}

\noindent Note that the homotopy class of $\gamma_{P,\epsilon}^{(\lambda)}$ depends neither on the choice of $P^\prime$ nor on the path connecting $P^\prime$ and $P$. Hence we can even construct a smoothly embedded  path in that way. By construction, this path covers arbitrarily long segments of the leaf $L_P^{(\lambda)}$ if $\epsilon$ is small enough.

\noindent Next, set $C_1 := \{ (e^{it}, 1) : t \in \mathbb{R} \}$ and $C_2 := \{ (1, e^{is}) : s \in \mathbb{R} \}$. The closed curves $C_1, C_2 \subset T^2$ represent generators $[C_1], [C_2] \in H_1(T^2, \mathbb{Z})$ intersecting exactly once in the point $(1,1) \in T^2$. Let $[\gamma_{P,\epsilon}^{(\lambda)}] \in H_1(T^2, \mathbb{Z})$ denote the homological $1$-class represented by $\gamma_{P,\epsilon}^{(\lambda)}$, and consider the topological intersection numbers $[\gamma_{P,\epsilon}^{(\lambda)}] \cdot [C_i]$ (see~\cite[14.6]{SZAT}).

\noindent \textit{Claim:} $\lambda = \lim_{\epsilon \rightarrow 0} \frac{[\gamma_{P,\epsilon}^{(\lambda)}] \cdot [C_2]}{[\gamma_{P,\epsilon}^{(\lambda)}] \cdot [C_1]}$.
\begin{proof}
We calculate the intersection numbers using their differential-topological interpretation, for smoothly embedded paths $\gamma_{P,\epsilon}^{(\lambda)}$ (see~\cite[5.2]{Hir94}). Since $L_P^{(\lambda)}$ intersects $C_1$ and $C_2$ everywhere with the same orientation, we just need to count the intersection points in $\gamma_{P,\epsilon}^{(\lambda)} \cap C_i$. Assuming for the moment that $P \not\in C_1 \cup C_2$, for small enough $\epsilon$ we only need to count the intersection points of the part of $\gamma_{P,\epsilon}^{(\lambda)}$ lying on $L_P^{(\lambda)}$ with $C_i$. This part is the image $p([0, b_\epsilon])$ of an interval $[0, b_\epsilon] \subset \mathbb{R}$ under the parametrisation $p: \mathbb{R} \rightarrow L_P^{(\lambda)}$ introduced above. Then
\[ \left[ \frac{\lambda b_\epsilon}{2\pi}\right] \leq \left| p([0, b_\epsilon]) \cap C_2 \right| \leq  \left[ \frac{\lambda b_\epsilon}{2\pi}\right] + 1\ \mathrm{and} \left[ \frac{b_\epsilon}{2\pi}\right] \leq \left| p([0, b_\epsilon]) \cap C_1 \right| \leq  \left[ \frac{b_\epsilon}{2\pi}\right] + 1, \]
where $[x]$ denotes the maximal integer $\leq x \in \mathbb{R}$ and $ \left| p([0, b_\epsilon]) \cap C_i \right|$ the number of intersection points of $p([0, b_\epsilon])$ and $C_i$. As discussed above, $b_\epsilon \rightarrow \infty$ if $\epsilon \rightarrow 0$, and the claim follows.

\noindent If $P \in C_1 \cup C_2$ the path in $\gamma_{P,\epsilon}^{(\lambda)}$ connecting $P^\prime$ with $P$ can be chosen to intersect $C_i$ only in a number of points bounded from above independently of $\epsilon$. Hence the claim also holds in that case.
\end{proof}

\noindent Now, we calculate:
\[ \lambda = 
\lim_{\epsilon \rightarrow 0} \frac{[\gamma_{P,\epsilon}^{(\lambda)}] \cdot [C_2]}{[\gamma_{P,\epsilon}^{(\lambda)}] \cdot [C_1]} = \lim_{\epsilon \rightarrow 0} \frac{[\phi(\gamma_{P,\epsilon}^{(\lambda)})] \cdot [\phi(C_2)]}{[\phi(\gamma_{P,\epsilon}^{(\lambda)})] \cdot [\phi(C_1)]} = 
\lim_{\epsilon \rightarrow 0} \frac{[\phi(\gamma_{P,\epsilon}^{(\lambda)})] \cdot [C_2]}{[\phi(\gamma_{P,\epsilon}^{(\lambda)})] \cdot [C_1]}, \]
by the Claim and since $\phi:T^2 \rightarrow T^2$ is a homeomorphism assumed to be homotopic to the identity. But $\phi(L_P^{(\lambda)}) = L_{\phi(P)}^{(\mu)}$, hence $\phi(\gamma_{P,\epsilon}^{(\lambda)})$ is a path constructed as above for the leaf $L_{\phi(P)}^{(\mu)}$ of $F_\mu$ and the neighborhood basis $U_\epsilon := \phi(B_P(\epsilon))$ of $\phi(P)$, so the above limit is equal to
\[ \lim_{\epsilon \rightarrow 0} \frac{[\gamma_{\phi(P),U_\epsilon}^{(\mu)}] \cdot [C_2]}{[\gamma_{\phi(P),U_\epsilon}^{(\mu)}] \cdot [C_1]} = \mu, \]
once again by the Claim.
\end{proof}

\begin{thm}
Two holomorphic foliation germs $\mathcal{F}_\lambda, \mathcal{F}_\mu$, $\mu, \lambda \in \mathbb{R}_{>0} - \mathbb{Q}_{>0}$, are topologically equivalent if, and only if $\lambda = \mu$ or $= \frac{1}{\mu}$.
\end{thm}
\begin{proof}
By the Reconstruction Theorem~\ref{topeq-intersec-fol-thm} it is enough to show the statement for the intersection foliations $\mathcal{F}_\lambda \cap S^3$ and $\mathcal{F}_\mu \cap S^3$. 

\noindent Exchanging the coordinates yields a topological equivalence $\Phi$ of $\mathcal{F}_\lambda \cap S^3$ with $\mathcal{F}_{\frac{1}{\lambda}} \cap S^3$. On the other hand, let $\Phi$ be a topological equivalence of $\mathcal{F}_\lambda \cap S^3$ with $\mathcal{F}_\mu \cap S^3$. As above, for $0 < \epsilon_x < 1$ the restriction $\Phi_{|T_x(\epsilon_x)}$ maps the torus $T_x(\epsilon_x)$ to another torus $T_x(\epsilon_x^\prime)$ and induces a topological equivalence of the Kronecker foliations $F_\lambda = \mathcal{F}_{\lambda|T_x(\epsilon_x)}$ and $F_\mu = \mathcal{F}_{\mu|T_x(\epsilon_x^\prime)}$.
Prop.~\ref{map-class-prop} shows that $\Phi_{|T_x(\epsilon_x)}$ must be of type $\left( \begin{array}{cc} \pm 1 & 0 \\ 0 & \pm 1 \end{array} \right)$ or $\left( \begin{array}{cc} 0 & \pm 1 \\ \pm 1 & 0 \end{array} \right)$ in the mapping class group of a $2$-dimensional torus. Then Thm.~\ref{Kron-hom-thm} implies that $\lambda = \mu$ or $\lambda = \frac{1}{\mu}$.
\end{proof}

\section{Topological equivalence classes in dimension $2$} \label{class-dim2-sec}

\noindent In each of the sections~\ref{nonreal-sec}, \ref{resonant-sec}, \ref{rational-ssec} and \ref{irrational-ssec} we identified the topological equivalence classes of plane holomorphic foliation germs represented by vector fields of a certain type, and the list in Rem.~\ref{dim2-fol-rem} shows that every plane holomorphic foliation germ is of one of these types. Consequently, the classification is completed by the following statement:
\begin{thm}
The topological equivalence classes determined in sections~\ref{nonreal-sec}, \ref{resonant-sec}, \ref{rational-ssec} and \ref{irrational-ssec} are pairwise distinct. 
\end{thm} 
\begin{proof}
If the eigenvalues of the linear part of the representing vector field are $\mathbb{R}$-linearly independent then there exists two closed leaves in the intersection foliation, and the closure of any other leaf of the intersection foliation consists of the leaf and these two closed leaves -- see the results of Section~\ref{nonreal-sec}. If the eigenvalues are $\mathbb{R}$-linearly dependent amd have resonances then there is only one closed leaf in the intersection foliation -- see the results of Section~\ref{resonant-sec}. If the eigenvalues are $\mathbb{Q}$-linearly dependent but the vector field is non-resonant then every leaf in the intersection foliation is closed -- see the results of Section~\ref{rational-ssec}. Finally, if the eigenvalues are $\mathbb{R}$-linearly dependent but $\mathbb{Q}$-linearly independent then all leaves in the intersection foliation besides the two closed leaves have as closure a torus -- see the results of Section~\ref{irrational-ssec}. 

\noindent Thus, in each of the four cases, there exist leaves of the intersection foliation with topological properties not occuring in the other cases. Hence the Reconstruction Theorem~\ref{topeq-intersec-fol-thm} shows the theorem.
\end{proof}

\section{Topological equivalence classes in dimension $\geq 3$} \label{higher-dim-sec}

\noindent Guckenheimer's Stability Theorem generalizes Thm.~\ref{R-LI-thm} to arbitrary dimensions:
\begin{thm}[{\cite{Guc72}}] \label{Guck-thm}
Let $\sum_{i=1}^n \lambda_i z_i \frac{\partial}{\partial z_i}$ and $\sum_{i=1}^n \mu_i z_i \frac{\partial}{\partial z_i}$ represent two holomorphic foliation germs with an isolated singularity in $0 \in \mathbb{C}^n$ such that $\lambda_1, \ldots, \lambda_n$ resp. $\mu_1, \ldots, \mu_n$ are in the Poincar\'e domain and pairwise $\mathbb{R}$-linearly independent. Then $\mathcal{F}_1$ and $\mathcal{F}_2$ are topologically equivalent. 
\end{thm}

\noindent Guckenheimer also showed that $\mathcal{F}_1$ and $\mathcal{F}_2$ are topologically equivalent if, under the same assumptions on the $\lambda_i$, the vector field $\theta_2$ representing $\mathcal{F}_2$ is obtained from $\sum_{i=1}^n \lambda_i z_i \frac{\partial}{\partial z_i}$ representing $\mathcal{F}_1$ by a sufficiently small holomorphic perturbation. This implies the following classification result:
\begin{prop}
Let $\mathcal{F}_1$ and $\mathcal{F}_2$ be two holomorphic foliation germs of rank $1$ with an isolated singularity in $0 \in \mathbb{C}^n$ represented by $[U_1, \theta_1]$ and $[U_2, \theta_2]$ such that the eigenvalues of the linear parts of the vector fields $\theta_1$ resp. $\theta_2$ are in the Poincar\'e domain and pairwise $\mathbb{R}$-linearly independent. Then $\mathcal{F}_1$ and $\mathcal{F}_2$ are topologically equivalent.
\end{prop}
\begin{proof}
Assume that $\theta_1 = \sum_{i=1}^n f_i(z) \frac{\partial}{\partial z_i}$ and $\theta_2 = \sum_{i=1}^n g_i(z) \frac{\partial}{\partial z_i}$. As discussed in Section~\ref{prelim-sec} we can assume that the non-linear terms of the power series $f_i(z)$ and $g_i(z)$ consist of resonant monomials $z_1^{m_1} \cdots z_n^{m_n}$ wrt the eigenvalues $\lambda_1, \ldots, \lambda_n$ of the linear part of $\theta_1$ resp.\ $z_1^{n_1} \cdots z_n^{n_n}$ wrt the eigenvalues $\mu_1, \ldots, \mu_n$ of the linear part of $\theta_n$, that is, the $\lambda_1, \ldots, \lambda_n$ resp. $\mu_1, \ldots, \mu_n$ satisfy the resonance $\lambda_i = \sum_{j=1}^n m_j \lambda_j$ resp.\ $\mu_i = \sum_{j=1}^n n_j \mu_j$ for some integers $m_j, n_j \geq 0$. Since $\lambda_1, \ldots, \lambda_n$ resp.\ $\mu_1, \ldots, \mu_n$ are in the Poincar\'e domain there are only finitely many of these resonant monomials, hence $f_i(z)$ and $g_i(z)$ are polynomials.

\noindent Possibly after a holomorphic coordinate change we can furthermore assume that the real parts of all the $\lambda_i$ and $\mu_i$ are positive and that
\[ 0 < \mathrm{Re} \lambda_1 < \cdots < \mathrm{Re} \lambda_n\ \mathrm{resp.\ } 
   0 < \mathrm{Re} \mu_1 < \cdots < \mathrm{Re} \mu_n.  \]
Thus, resonances $\lambda_i = \sum_{j=1}^n m_j \lambda_j$ resp.\ $\mu_i = \sum_{j=1}^n n_j \mu_j$ always satisfy $m_j = n_j = 0$ for $j \geq i$. Consequently, rescaling the $i$th coordinate $z_i$ by a real factor $\epsilon_i$ such that $0 < \epsilon_1 \ll \epsilon_2 \ll \cdots \ll \epsilon_n$ changes the vector fields $\theta_1, \theta_2$ to vector fields with non-linear parts arbitarily close to $0$.

\noindent So Guckenheimer's Stability Theorem implies that $\mathcal{F}_1$ resp.\ $\mathcal{F}_2$ are topologically equivalent to the foliations represented by the linear parts $\sum_{i=1}^n \lambda_i z_i \frac{\partial}{\partial z_i}$ resp. $\sum_{i=1}^n \mu_i z_i \frac{\partial}{\partial z_i}$ of $\theta_1$ resp.\ $\theta_2$, and these foliations are topologically equivalent by Thm.~\ref{Guck-thm}. 
\end{proof}

\noindent Under the assumptions of the proposition the appearence of resonant monomials involving only $\mathbb{R}$-linearly independent eigenvalues does not influence the topological equivalence class. So for more general situations we introduce the following notion:
\begin{Def}
Let $(\lambda_1, \ldots, \lambda_n) \in \mathbb{C}^n$ be a set of complex numbers in the Poincar\'e domain. A resonance $\lambda_i = \sum_{j=1}^n m_j \lambda_j$ is called inessential if not all the $\lambda_j \in \mathbb{C}$ with $m_j \neq 0$ lie on the same real ray starting in the origin. Otherwise the resonance is called essential.
\end{Def} 

\noindent The $2$-dimensional classification in Sections~\ref{nonreal-sec} - \ref{class-dim2-sec} shows that $\mathbb{R}$-linear (in)dependence of the two eigenvalues of the linear part of a representing vector field distinguishes the topological equivalence class of holomorphic foliation germs of rank $1$ with an isolated singularity in $0 \in \mathbb{C}^2$ of Poincar\'e type. In higher dimension we extend this dichotomy to the following invariant:
\begin{Def}
The ray configuration of a tuple $(\lambda_1, \ldots, \lambda_n) \in \mathbb{C}^n$ is the ordered partition of this set into subsets consisting of those $\lambda_i \in \mathbb{C}$ lying on the same real ray starting in the origin, and the subsets are ordered by increasing angle of this ray with the positive real axis.

\noindent Two ray configurations are called equivalent if the sizes of the partition subsets, in the order of the partition, are equal, or become equal after reversing the order of one of the partitions. 
\end{Def}

\noindent Finally, the $2$-dimensional classification shows that topologically equivalent plane holomorphic foliation germs of rank $1$ with an isolated singularity in $0 \in \mathbb{C}^2$ of Poincar\'e type having equivalent ray configurations are also holomorphically equivalent. 

\noindent Combining all these observations we predict the following behaviour of such foliation germs in arbitrary dimensions:
\begin{conj}
Two holomorphic foliation germs of rank $1$ with an isolated singularity in $0 \in \mathbb{C}^n$ of Poincar\'e type are topologically equivalent if and only if the following two conditions are satisfied:
\begin{itemize}
\item[(1)] The ray configurations of the tuples of eigenvalues of the linear part of a vector field representing the foliation germs are equivalent.
\item[(2)] For every two corresponding partition subsets $\{i_1, \ldots, i_k\}$, $\{j_1, \ldots, j_k\} \subset \{1, \ldots, n\}$ of the two ray configurations, the restrictions of the two foliation germs to the linear subspaces
\[ L_1 := \left\{ z_l = 0: l \neq i_1, \ldots, i_k \right\}, 
    L_2 := \left\{ z_m = 0: m \neq j_1, \ldots, j_k \right\} \subset \mathbb{C}^n \] 
are holomorphically equivalent.
\end{itemize}
\end{conj}

\noindent In particular, the conjecture predicts in full generality that the appearence of inessential resonant monomials does not influence the topological equivalence class.

%\bibliographystyle{alpha}
%\bibliography{/Forschung/Mathematik/doktor}
%\bibliography{doktor}

\def\cprime{$'$}

\end{document}